\documentclass[11pt,reqno]{amsart}
\usepackage{amssymb,amsfonts,amsthm,amscd,stmaryrd,dsfont,esint,upgreek,constants,todonotes}
\usepackage{pdfsync}
\usepackage{amsmath, constants}\multlinegap=0pt
\usepackage[active]{srcltx}
\usepackage[colorlinks=true,urlcolor=blue,
citecolor=red,linkcolor=blue,linktocpage,pdfpagelabels,bookmarksnumbered,bookmarksopen]{hyperref}
\usepackage[latin1]{inputenc}
\usepackage{amsthm}
\usepackage{amssymb}
\usepackage{bm}
\usepackage[francais,english]{babel}
\usepackage[francais,english]{babel}

\newcommand{\s}[1]{{\mathcal #1}}

\newcommand{\bb}[1]{{\mathbb #1}}

\newtheorem{theorem}{Theorem}[section]

\newtheorem{lemma}[theorem]{Lemma}
\newtheorem{proposition}[theorem]{Proposition}
\newtheorem{problem}[theorem]{Problem}
\newtheorem{definition}[theorem]{Definition}

\newtheorem{remark}[theorem]{Remark}

\numberwithin{equation}{section}
\numberwithin{theorem}{section}

\definecolor{Red}{cmyk}{0,1,1,0.2}

\def\ds{\displaystyle}
\def\ep{\varepsilon}

\newcommand{\dd}{\,{\rm d}}

\newcommand{\id}{\rm{id}}

\newcommand{\diver}{\rm{div}}
\def\Om{\Omega}

\newcommand{\cA}{{\mathcal A}}
\newcommand{\cB}{{\mathcal B}}

\newcommand{\cL}{{\mathcal L}}
\newcommand{\cK}{{\mathcal K}}

\newcommand{\cP}{{\mathcal P}}

\newcommand{\fM}{{\mathfrak M}}

\newcommand{\Z}{{\mathbb Z}}

\newcommand{\weaks}{\stackrel{*}{\rightharpoonup}}
\newcommand{\ov}{\overline}

\newcommand{\R}{\mathbb{R}}
\newcommand{\T}{\mathbb{T}}
\renewcommand{\O}{\Omega}
\renewcommand{\rm}{\mathrm}

\newcommand{\et}{{\bm{\eta}}}

\newcommand{\mres}{\mathbin{\vrule height 1.6ex depth 0pt width
0.13ex\vrule height 0.13ex depth 0pt width 1.3ex}}

\def\a{\alpha}
\def\b{\beta}
\def\d{\delta}
\def\g{\gamma}
\def\l{\lambda}

\def\sig{\sigma}

\def\e{\varepsilon}

\def\rg{\rangle} 
\def\lg{\langle} 
\def\ds{\displaystyle}
\def\T{\bb{T}}
\newcommand{\be}{\begin{equation}}
\newcommand{\ee}{\end{equation}}

\title{First order Mean Field Games with density constraints: Pressure equals Price}

\author[P. Cardaliaguet]{Pierre Cardaliaguet}
\address{Ceremade, Universit\'e Paris-Dauphine,
Place du Mar\'echal de Lattre de Tassigny, 75775 Paris cedex 16 - France}
\email[P. Cardaliaguet]{cardaliaguet@ceremade.dauphine.fr }

\author[A.R. M\'esz\'aros]{Alp\'ar R. M\'esz\'aros}
\address{Department of Mathematics, University of California, Los Angeles - USA}
\email[A.R. M\'esz\'aros]{alpar@math.ucla.edu}

\author[F. Santambrogio]{Filippo Santambrogio}
\address{Laboratoire de Math\'ematiques d'Orsay, Univ. Paris-Sud, CNRS, Universit\'e Paris-Saclay, 91405 ORsay cedex - France}
\email[F. Santambrogio]{filippo.santambrogio@math.u-psud.fr}

\date{\today}
\begin{document}

\maketitle
\begin{abstract}
In this paper we study Mean Field Game systems under density constraints as optimality conditions of two optimization problems in duality. A weak solution of the system contains an extra term, an additional price imposed on the saturated zones. We show that this price corresponds to the pressure field from the models of incompressible Euler's equations {\it \`a la Brenier}. By this observation we manage to obtain a minimal regularity, which allows to write optimality conditions at the level of single agent trajectories and to define a weak notion of Nash equilibrium for our model. 
\end{abstract}

\section{Introduction}

\subsection{The MFG system} 
Introduced by Lasry-Lions \cite{LL06cr1, LL06cr2, LL07mf} (see also Huang-Malhamé-Caines \cite{HCMieeeAC06}) the mean field game system (in short, MFG system) describes a differential game with infinitely many identical players who interact through their repartition density. The first order MFG system with a local coupling takes the form
\be\label{MFG0}
\left\{
\begin{array}{crcll}
{\rm{(i)}} & -\partial_t u + H(x,D u) &=& f(x,m) &  {\rm{in}}\; (0,T)\times \T^d \\
{\rm{(ii)}}& \partial_t m - \mathrm{div} \left(mD_p H(x,D u)\right) &=& 0  & {\rm{in}}\; (0,T)\times \T^d\\
{\rm{(iii)}} & u(T,x) =  g(x),&\;&  m(0,x) = m_0(x)  & {\rm{in}}\; \T^d.\\
\end{array}\right.
\ee
Here, to avoid the discussion of the boundary data, we work for simplicity with periodic boundary conditions, i.e., in the torus $\T^d:=\R^d/\Z^d$. Since the main accent in this paper will be on the modeling of the density constraint, we keep this simpler setting. Let us remark that without much effort, with the same ideas it is possible to treat the case of general domains with the corresponding boundary conditions. The Hamiltonian $H:\T^d\times \R^d\to \R$ is typically convex with respect to the last variable and the coupling cost $f:\T^d\times [0,+\infty)$ is nondecreasing with respect to the last variable. The monotonicity of the coupling formalizes the idea that the players dislike congested areas. It will be highly exploited later in the variational setting, which will imply in particular a convexity property for the energy functional. Moreover all these assumptions are typical in the general
MFG theory.
 
Let us briefly describe the interpretation of \eqref{MFG0}. In the above backward-forward system, $u=u(t,x)$ is the value function associated to any tiny player while $m=m(t,x)$ is the density of the players at time $t$ and at position $x$. The value function $u(t,x)$ is formally given by
$$
u(t,x)= \inf_\gamma
\int_t^T L(\gamma(s), \dot\gamma(s))+ f(\gamma(s), m(s,\gamma(s))\dd s +g(\gamma(T))
$$
where the player minimizes over the paths $\gamma:[t,T]\to \T^d$ with $\gamma(t)=x$, $f=f(x,m(t,x))$ is the running cost, $L$ is obtained from the Fenchel conjugate of $H$ with respect to the last variable, and $g:\T^d\to \R$ is the terminal cost at the terminal time $t=T$. The running cost $f$ couples the two equations and acts as a penalization for those regions where the density $m$ is too high. 

At the initial time $t=0$, the initial distribution is $m_0$ (a probability measure on $\T^d$). Then the density evolves according to the motion of the players. Since -- by standard argument in optimal control -- it is optimal for the players to play $\dot \gamma(s)= -D_pH(\gamma(s), Du(s,\gamma(s))$, the evolution of the density is given by the continuity equation \eqref{MFG0}-(ii). 

Note that each tiny player acts as if he/she knew the evolution of the players' density $m=m(t,x)$ (he/she somehow ``forecasts" it, as usually in a ``rational expectations'' framework). Actually he/she needs this forecast in order to solve his/her individual control problems. Solving this problem he/she obtains the value function $u$ and the optimal velocity field  $-D_pH(\cdot, Du).$ Then the ``true'' evolution of the players' density is given as the transport of the initial density by this field (this corresponds to the continuity equation in system \eqref{MFG0}). 

The mean field game system corresponds to an equilibrium situation where the ``forecast" of the players is correct: the solution of the continuity equation is indeed $m=m(t,x)$, which was the forecast made by the players. In terms of game theory, this corresponds to a Nash equilibrium. 

Existence and uniqueness of solutions for the above problem are discussed by P.-L. Lions in \cite{LLperso} (through a reduction to an elliptic equation in time-space when the coefficients are smooth, under the additional assumption $\lim_{m\to 0} f(x,m)=-\infty$, which is satisfied for instance for log-like couplings, which guarantees $m>0$ and hence ellipticity) and in Cardaliaguet \cite{carda}, Graber \cite{graber}, Cardaliaguet-Graber \cite{cargra}, Cardaliaguet-Porretta-Tonon \cite{carporton} (following an approach by variational methods suggested in \cite{LL07mf} and also inspired by Benamou-Brenier \cite{bb}). Recently, in \cite{BenCar} Benamou and Carlier used similar variational techniques to study an augmented Lagrangian scheme for MFG problems and obtain efficient numerical simulations.

\subsection{The problem with a density constraint}
In this paper we study the behavior of the MFG system when there is a {\it density constraint}, i.e., when the density $m$ cannot exceed some given value $\ov m>1/|\T^d|=1$. Namely: $0\leq m(t,x)\leq \ov m$ at any point $(t,x)$. In other words, the players pay an infinite price when the density goes above $\ov m$: $f(x,m)=+\infty$ if $m>\ov m$. The question of how to model this situation was first introduced by Santambrogio \cite{San} and then investigated M\'esz\'aros and Silva in \cite{MesSil} in the framework of stationary second order models.  We emphasize the fact that imposing a density constraint will result in a so-called ``hard congestion'' effect in the model. Models of MFGs where so-called ``soft congestion'' (meaning that agents slow down when they reach zones with high density) effects have been studied recently by Gomes and Mitake in \cite{GomMit}, by Gomes and Voskanyan in \cite{GomVos} and by Burger, Di Francesco, Markowich and Wolfram in \cite{BurDiFMarWol}.  

Coming back to our model, there are several issues in the interpretation of system \eqref{MFG0} when there is a density constraints. Indeed, the above interpretation does not make sense anymore for the following reason: if, on the one hand, the constraint $m\leq \ov m$ is fulfilled, then the  minimization problem of the agents (due to the fact that they are considered negligible against the others) does not see this constraint and the pair $(u,m)$ is the solution of a standard MFG system; but this solution has no reason to satisfy the constraint, and there is a contradiction. On the other hand, if there are places where $m(t,x)>\ov m$, then the players do not go through these places because their cost is infinite there: but then the density at such places is zero, and there is  again a contradiction. So, in order to understand the MFG system with a density constraint, one has to change the point of  view. We shall see that there are several ways to understand more deeply the phenomena behind this question. We warn the reader that the model that we will obtain significantly differs from that in \cite{San}.

Perhaps the simplest approach  is to go through an approximation argument: let us consider the solution $(u^\ep,m^\ep)$ corresponding to a running cost $f^\ep$ which is  finite everywhere, but tends to infinity as $\ep$ tends to $0$ when $m>\ov m$. In orther words, $f^\ep(x,m)\to f(x,m)$ if $m\leq \ov m$ and $f(x,m)\to +\infty$ if $m>\ov m$, as $\e\to 0$. In this case the MFG system with a density constraint should simply be the limit configuration (a limit which should be proven to be well-defined). 

We indeed show that the pair $(u^\ep, m^\ep)$ has (up to subsequences) a limit $(u,m)$ which satisfies (in a weak sense) the following system: 
\be\label{MFGdensity0}
\left\{
\begin{array}{crcll}
{\rm{(i)}} & -\partial_t u(t,x) + H(x,D u(t,x)) &=& f(x,m(t,x))+\beta(t,x) &  {\rm{in}}\; (0,T)\times \T^d \\
{\rm{(ii)}}& \partial_t m(t,x) - \mathrm{div} \left(mD_p H(x,D u(t,x))\right) &=& 0  & {\rm{in}}\; (0,T)\times \T^d\\
{\rm{(iii)}} & u(T,x) =  g(x)+\beta_T(x),&\;&  m(0,x) = m_0(x)  & {\rm{in}}\; \T^d\\
{\rm{(iv)}} & 0\leq m(t,x) \leq \overline m & & & {\rm{in}}\; [0,T]\times \T^d
\end{array}\right.
\ee
Besides the expected density constraint (iv), two extra terms appear: $\beta$ in (i) and $\beta_T$ in (iii). These two quantities turn out to be nonnegative and concentrated on the set $\{m=\ov m\}$. They formally correspond to an extra price payed by the players to go through zones where the concentration is saturated, i.e., where $m=\ov m$ (in traffic language, this would be a {\it toll}).  In other words, the new optimal control problem for the players is now (formally)
\begin{equation}\label{new_control2}
\begin{array}{c}
\ds u(t,x)= \inf_{\begin{subarray}{c}
\gamma\\
\g(t)=x
\end{subarray}}
\int_t^T L(\gamma(s), \dot\gamma(s))+ f(\gamma(s), m(s,\gamma(s))+\beta(s, \gamma(s))\dd s +g(\gamma(T))+\beta_T(\gamma(T)),
\end{array}
\end{equation}
and thus (still formally) satisfies the dynamic programming principle: for any $0\le t_1\le t_2< T,$
\begin{equation}\label{new_control}
\begin{array}{c}
\ds u(t_1,x)= \inf_{\begin{subarray}{c}
\gamma\\
\g(t_1)=x
\end{subarray}}
\int_{t_1}^{t_2} L(\gamma(s), \dot\gamma(s))+ f(\gamma(s), m(s,\gamma(s))+\beta(s, \gamma(s))\dd s +u(t_2,\gamma(t_2)).
\end{array}
\end{equation}
The ``extra prices'' $\beta$ and $\beta_T$ discourage too many players to be attracted by the area where the constraint is saturated, thus ensuring the density constraints (iv) to be fulfilled. The reader familiar with theoretical economics can realize immediately that this is exactly the typical role of prices: a price is a quantity determined by a global configuration, which replaces, in the individual choices of the agents, the presence of the constraint.

\subsection{The variational method} 
Another way to see the problem is the following: it is known (see \cite{LL07mf}) that the solution $(u,m)$ to \eqref{MFG0} can be obtained by variational methods at least when $f$ is finite everywhere. More precisely, the value function $u$ is (formally) given as a minimizer of the functional 
$$
{\mathcal A}(u):= \int_0^T \int_{\bb{T}^d} F^*(x,-\partial_t  u + H(x,D u))\dd x\dd t - \int_{\bb{T}^d}  u(0,x)\dd m_0(x),
 $$
subject to the constraint that $u(T,x)=g(x)$, where $F=F(x,m)$ is an antiderivative of $f=f(x,m)$ with respect to $m$ and $F^*$ is its Legendre-Fenchel conjugate w.r.t. the second variable. In the same way $m$ is (formally) given as a minimizer of the problem 
 $$
{\mathcal B}(m,w):=  \int_{\bb{T}^d}  g(x)m(T,x)\dd x + \int_0^T \int_{\bb{T}^d} m(t,x)H^*\left(x,-\frac{w}{m}\right) + F(x,m(t,x))\, \dd x \dd t
 $$
subject to the constraint 
$$
\partial_t m + \mathrm{div}  (w)= 0\;\;\; {\rm{in}}\; (0,T)\times \T^d, \qquad m(0)=m_0,
$$
 where $H^*$ is the convex conjugate of $H$ with respect to the last variable. With the language of the theory of optimal control of PDEs, the additional variable $w:[0,T]\times\T^d\to\R^d$ plays the role of the control, while $m$ can be seen as the state variable. 
 
It turns out that both problems make perfectly sense, even when $f(x,m)=+\infty$ if $m>\ov m$. In fact, if $f^\ep$ is a finite approximation of $f$ as before, one can expect the minimizers of ${\mathcal A}^\ep$ and ${\mathcal B}^\ep$ (corresponding to $f^\ep$) to converge to the minimizers of ${\mathcal A}$ and ${\mathcal B}$ as $\ep\to0$ (as a simple consequence of $\Gamma-$convergence). This is precisely what happens. Note that, as $f(x,m)=+\infty$ for $m>\ov m$, $F(x,m)$ has the same property, so that $F^*(x,m)$ is linear on $[\ov m,+\infty)$. This linear behavior explains the appearance of the terms $\beta$ and $\beta_T$ described above. 
 
\subsection{Connections between MFGs with density constraints and the incompressible Euler's equations \`a la Brenier} 
It is not surprising, due to the constraint $m\leq \ov m$, that some strong connections between our model and the variational models for the incompressible Euler's equations studied by Brenier (see \cite{br}) and also by Ambrosio and Figalli (see \cite{af}) arise. What was unexpected at the beginning of our study is the role that this connection would play in regularity. In order to understand the analogy, notice that the incompressibility constraint in the model of Brenier to study perfect fluids is what introduces the pressure field. Morally the same effect happens imposing density  constraint for MFG. Using the common variational structure, similar also to the one introduced by Benamou and Brenier in \cite{bb}, shared by the incompressible Euler equation and by our model, we can easily interpret the terms $\b$ and $\b_T$, that we call ``additional prices/costs'' for the agents (appearing only if they pass through saturated zones) in \eqref{MFGdensity0} as a sort of pressure field from fluid mechanics. This observation motivates the title of our work as well.

Using similar techniques as in \cite{br} and \cite{af, af2} we show that $\b$ is an $$L^2_{\rm{loc}}((0,T);BV(\T^d))\hookrightarrow L^{d/(d-1)}_{\rm{loc}}((0,T)\times\T^d)$$ function (while a priori it was only supposed to be a measure) and $\b_T$ is $L^1(\T^d)$. With the help of an example we show that this local integrability cannot be extended so as to include the final time $t=T$, which shows that the result is somehow sharp. This regularity property will allow us to give a clearer (weak) meaning to the control problem \eqref{new_control}, obtaining optimality conditions along single agent trajectories. Our techniques to proceed with the analysis rely on the properties of measures defined on paths, that we shall call {\it density-constrained flows} in our context, and we are exploiting some properties of a Hardy-Littlewood type maximal functional as well (this is very much inspired by \cite{af}).

After this analysis we deduce the existence of a {\it local weak  Nash equilibrium} for our model. 

\vspace{0.5cm}

The paper is organized in the following way. We first introduce our main notation and assumptions (Section \ref{sec:Hyp}). Then we discuss the two optimization problems for ${\mathcal A}$ and ${\mathcal B}$ described above (Section \ref{sec:optipb}). We introduce the definition of the MFG system with a density constraint, present our main existence result as well as the approximation by standard MFG systems in Section \ref{sec:MFGdensity}. In Section \ref{sec:example}, by means of an example, we study some finer properties of a solution $(m,u,\b,\b_T)$ of the MFG system with density constraints. Section \ref{sec:regularity} is devoted to the proof of the $L_{\rm{loc}}^{d/(d-1)}$ integrability of the additional price $\b$ under some additional assumptions on the Hamiltonian and the coupling. Finally, having in hand this integrability property,  we introduce   in Section \ref{sec:nash}  the optimal density-constrained flows and derive optimality conditions along single agent paths, which allow in particular to study the existence of the local weak Nash equilibrium.\\

{\bf Acknowledgement:} The first author was partially supported by the ANR (Agence Nationale de la Recherche) projects  ANR-10-BLAN 0112, ANR-12-BS01-0008-01 and ANR-14-ACHN-0030-01. The second and third authors were partly supported by the ANR project ANR-12-MONU-0013 and by the iCODE project ``strategic crowds'' of the IDEX {\it Universit\'e Paris-Saclay}.

\section{Notations, assumptions and preliminaries}\label{sec:Hyp}

We consider the MFG  system with a density constraint \eqref{MFGdensity0}
under the assumption that all the maps are periodic in space. Typical conditions are 
\begin{enumerate}
\item[{\bf(H1)}] The density constraint $\ov m$ is larger than $1=1/|\T^d |$. 
\item[{\bf(H2)}] (Conditions on the initial and final conditions) $m_0$ is a probability measure on $\bb{T}^d$ which is absolutely continuous with respect to Lebesgue measure and there exists $\ov c>0$ such that $0\leq m_0< \ov m-\ov c$ a.e. on $\bb{T}^d.$ 
We  assume that $ g:\T^d\to \R$ is a $C^1$ function on $\bb{T}^d$.
\item[{\bf(H3)}] (Conditions on the Hamiltonian) $H:\bb{T}^d \times \bb{R}^d \to \bb{R}$ is continuous in both variables, convex and differentiable in the second variable, with $D_pH$ continuous in both variables. Moreover, $H$ has superlinear growth in the gradient variable: there exist $r >1 $ and $C >0$ such that
\begin{equation}
\label{eq:hamiltonian_bounds}
\frac{1}{rC}|p|^r-C \leq H(x,p) \leq \frac{C}{r}|p|^r + C.
\end{equation}
We denote by $H^*(x,\cdot)$ the Fenchel conjugate of $H(x,\cdot)$, which, due to the above assumptions, satisfies
\begin{equation}
\label{eq:hamiltonian_conjugate_bounds}
\frac{1}{r'C}|q|^{r'}-C \leq H^*(x,q) \leq \frac{C}{r'}|q|^{r'} + C,
\end{equation}
where $r'$ is the conjugate of $r$. 
We will also denote by $L$ the Lagrangian given by $L(x,q) = H^*(x,-q)$, which thus satisfies the same bounds as $H^*$.
\item[{\bf(H4)}] (Conditions on the coupling) Let $f$ be continuous on $\bb{T}^d \times [0,\ov m]$, non-decreasing in the second variable with 
 $f(x,0)=0$. 
\end{enumerate}

Let us comment now on the previous assumptions. {\bf (H1)} and {\bf (H2)} are assumptions on the initial and final conditions, since $m_0$ is a probability measure on $\T^d,$ because of the mass-conservation $m_t$ will be also a probability measure on $\T^d$ for all $t\in[0,T].$ The density constraint should satisfy $\ov{m}>1,$ otherwise imposing that a probability measure $m\le\ov{m}$ on $\T^d$ would give either a trivial competitor or no competitor at all. It is natural to impose $0\le m_0\le\ov{m},$ i.e. we start with an initial distribution that already satisfies the constraint. The  modified upper bound $m<\ov{m}-\bar{c}$ for some $\bar{c}>0$ small real number is just a technical assumption that we need in the analysis. Assumptions {\bf (H3)} and {\bf (H4)} are natural growth and structural conditions that are typical while working with variational MFG systems. In particular imposing that $f$ is non-decreasing will imply that the energy functional $\cB$ is convex.

We define $F$ so that $F(x,\cdot)$ is an antiderivative of $f(x,\cdot)$ on $[0,\ov m]$, that is,
\begin{equation}
F(x,m) = \int_0^m f(x,s)\dd s, ~~ \forall ~ m \in [ 0, \ov m],
\end{equation}
and extend $F$ to $+\infty$ on $(-\infty,0)\times (\ov m,+\infty)$. 
It follows that $F$ is continuous on $\bb{T}^d \times [0,\ov m]$, is convex and differentiable in the second variable.
We also define $F^*(x,\cdot)$ to be the Fenchel conjugate of $F(x,\cdot)$ for each $x$. Note that
\be\label{remF*}
F^*(x,\alpha)\geq \alpha \ov m-F(x,\ov m)
\ee
and $F^*(\cdot,\a)=0$ for all $\a\le 0.$
Following the approach of Cardaliaguet-Carlier-Nazaret \cite{carcarnaz} (see also Cardaliaguet \cite{carda}, Graber \cite{graber} or Cardaliaguet-Graber \cite{cargra}) it seems that the solution to  \eqref{MFGdensity0} can be obtained as the system of optimality conditions for optimal control problems. 

\subsection{Optimal transport toolbox} In this subsection we collect some basic notions and results from the theory of optimal transportation which we will need in the sequel. We refer to \cite{OTAM, villani} general references to this theory. Let $\Om\subset\R^d$ be a compact subset (or any compact subset of a Polish space). Even if in the whole paper we will restrict ourselves to the case of $\Om=\T^d$, we state the following results in the general case. Given two probability measures $\mu,\nu\in \cP(\Om)$ and for $p\ge1$ we define the usual Wasserstein metric by means of the Monge-Kantorovich optimal transportation problem
$$W_p(\mu,\nu):=\inf\left\{ \int_{\Om\times\Om}|x-y|^p\,\dd\g(x,y)\;:\;\g\in\Pi(\mu,\nu)\right\}^{\frac1p},$$ 
where $\Pi(\mu,\nu):=\{\g\in\cP(\Om\times\Om):\;\; (\pi^x)_\#\g=\mu,\; (\pi^y)_\#\g=\nu\}$ and $\pi^x$ and $\pi^y$ denote the canonical projections from $\Om\times\Om$ onto $\Om$ (in a more general setting, $\Om$ being any compact subset of a Polish space, in the definition of $W_p$ one has to replace the Euclidean distance $|x-y|$ by the distance induced by the underlining metric $d$).  This quantity happens to be a distance on $\cP(\Om)$ which metrizes the weak-$\star$ convergence of probability measures; we denote by $\mathbb W_p(\Om)$ the space of probabilities on $\Om$ endowed with this distance. 

Historically, the quadratic case $p=2$ had been understood first. So, let us state the most fundamental results in this case. Under the additional assumption $\mu\ll\cL^d\mres\Om$ ($\mu$ is absolutely continuous w.r.t. the $d-$dimensional Lebesgue measure restricted to $\Om$) Brenier showed (see \cite{brenier1, brenier2}) that the optimal $\g$ in the above problem is actually induced by a map, which turns out to be  the gradient of a convex function, i.e. there exists $T:\O\to\O$ and $\psi:\Om\to\R$ convex such that $T=\nabla \psi$ and $\ov\g:=(\id,T)_\#\mu.$ The function $\psi$ is obtained as $\ds\psi(x)=\mbox{{\small $\frac 12$}} |x|^2-\varphi(x)$, where $\varphi$ is the so-called Kantorovich potential for the transport from $\mu$ to $\nu$, and is characterized as the solution of a dual problem. In this way, the optimal transport map $T$ can also be written as $T(x)=x-\nabla\varphi(x)$. Later, McCann  (see \cite{mccann}) introduced a useful notion of interpolation between probability measures: the curve $\mu_t:=\left( (1-t)x+ty\right)_\#\ov\g,$ for $t\in[0,1]$, gives a constant speed geodesic in the Wasserstein space connecting $\mu_0:=\mu$ and $\mu_1:=\nu.$

Based on this notion of interpolation, Benamou and Brenier using also some ideas from fluid mechanics, gave a dynamical formulation to the Monge-Kantorovich problem (see \cite{bb}). They showed that
$$\frac{1}{p}W_p^p(\mu,\nu)=\inf\left\{\cB_p(E_t,\mu_t)\; : \; \partial_t\mu_t+\rm{div}(E_t)=0,\; \mu_0=\mu,\; \mu_1=\nu \right\},$$ where $\cB_p:\fM([0,1]\times\Om)^d\times L^\infty([0,1];\mathbb W_p(\Om))\to \R\cup\{+\infty\}$\footnote{We denote by $\fM(X)$ the signed Radon measures on $X.$ Observe that $\mu\in L^\infty([0,1];\mathbb W_p(\Om))$ only means that $\mu=(\mu_t)_t$ is a time-dependent family of probability measures.} is given by
$$
\cB_p(E,\mu):=\left\{
\begin{array}{ll}
\ds\int_0^1\int_\Om \frac{1}{p}\left|\frac{\dd E}{\dd\mu}\right|^p\dd \mu_t(x)\dd t, &{\rm{if}}\ E\ll\mu,\\
+\infty, & {\rm{otherwise}}.
\end{array}
\right. 
$$
It is well-known that $\cB_p$ is jointly convex and l.s.c. w.r.t the weak-$\star$ convergence of measures (see Section 5.3.1 in \cite{OTAM}) and that, if $\partial_t \mu_t+\rm{div} (E_t)=0$, then $\cB_p(E,\mu)<+\infty$ implies that $t\mapsto \mu_t$ is a continuous curve, belonging to $W^{1,p}([0,1];\mathbb W_p(\Omega))$. In particular $\mu_t$ is well defined for all $t\in[0,1].$

We shall use the notion of {\it narrow convergence} in $\cP(\Om),$ which is the weak-$\star$ convergence in duality with continuous and bounded functions on $\Om.$ Since in our case $\Omega$ is compact, then $\cP(\Om)$ is also compact for this convergence. 

\subsection{Measures on curves and a superposition principle}
Let us denote by $\Gamma$ the set of absolutely continuous curves $\g:[0,T]\to\T^d.$ We denote by $\cP(\Gamma)$ the set of Borel probability measures defined on $\Gamma$. Let  us set $\cP_r(\Gamma)$ $(r\ge 1)$ to be the subset of $\cP(\Gamma)$ such that 
$$\int_\Gamma\int_0^T|\dot\g(s)|^r\dd s\dd\et(\g)<+\infty.$$
Note that the space $\Gamma$, which is naturally endowed with the uniform convergence topology, is not compact, and hence $\cP(\Gamma)$ is not compact for the narrow convergence. Yet, Prokhorov Theorem guarantees that any family of probability measures on a Polish space $\Om$ is relatively sequentially compact w.r.t. the narrow convergence provided it is {\it tight}. Tight means that for any $\e>0$ there exists a compact set $K\subseteq\Om$ such that for any $\mu$ from this family $\mu(\Om\setminus K)<\e.$ In the case of $\cP(\Gamma)$, it is easy to see that a uniform bound on $\ds\int_\Gamma\int_0^T|\dot\g(s)|^r\dd s\dd\et(\g)$ is enough to provide tightness. This will be useful later in the paper.

We also define the evaluation maps $e_t:\Gamma\to\T^d$, given by $e_t(\g):=\g(t)$ for all $t\in[0,T].$   This allows to state a well-known result, a connection between the solutions of the continuity equation and the measures on paths, called {\it superposition principle}, which can be considered as a weaker version of the DiPerna-Lions-Ambrosio theory (see for instance Theorem 8.2.1. from \cite{ags}). 

\begin{theorem}
Let $\mu:[0,T]\to\cP(\T^d)$ be a narrowly continuous solution of the continuity equation $ \partial_t\mu+\rm{div} (v\mu)=0,\; \mu_0\in\cP_2(\T^d)$ for a velocity field $v:(0,T)\times \T^d\to\R^d$ satisfying $\ds\int_0^T \int_{\T^d}|v_t|^2\dd\mu_t\dd t<+\infty.$ Then there exists $\et\in\cP(\Gamma)$ such that 
\begin{itemize}
\item[$\rm{(i)}$] $\mu_t=(e_t)_\#\et$ for all $t\in[0,T];$ 
\item[$\rm{(ii)}$] we have the energy inequality
$$\int_\Gamma\int_0^T|\dot\g(t)|^2\dd t\dd\et(\g)\le\int_0^T\int_{\T^d}|v|^2\dd\mu_t\dd t;$$
\item[$\rm{(iii)}$] $\dot\g(t)=v_t(\g(t)),\ $ for $\et-$a.e. $\g$ and a.e. $t\in[0,T].$
\end{itemize}
\end{theorem}

\section{Optimal control problems}\label{sec:optipb}

We look in this section at two PDE control problems that will lead to our Mean Field Game model.

The first one is an optimal control problem of Hamilton-Jacobi equations: 
denote by $\s{K}_P$ the set of functions $ u \in C^1([0,T] \times \bb{T}^d)$ such that $ u(T,x) =  g(x)$ (the subscript $P$ stands for ``primal''). Let us define on $\s{K}_P$ the functional
\begin{equation}
\s{A}( u) = \int_0^T \int_{\bb{T}^d} F^*(x,-\partial_t  u + H(x,D u))\dd x\dd t - \int_{\bb{T}^d}  u(0,x)\dd m_0(x) .
\end{equation}
Then we have our first optimal control problem. \medskip
\begin{problem}[Optimal control of HJ] \label{pr:smooth}
Find $\ds\inf_{ u \in \s{K}_P} \s{A}( u)$.
\end{problem} 
It is easy to check that one can restrict the optimization to the class of minimizers such that $-\partial_t  u + H(x,D u)\geq 0$, because $F^*(x,\alpha)= 0$ for $\alpha \leq 0$ (see Lemma 3.2  in \cite{carda}).  
\medskip

The second problem is an optimal control problem for the continuity equation: define $\s{K}_D$ to be the set of all pairs $(m, w) \in L^1([0,T] \times \bb{T}^d) \times L^1([0,T] \times \bb{T}^d;\bb{R}^d)$ such that $m \geq 0$ almost everywhere, $\ds\int_{\bb{T}^d}m(t,x)\dd x = 1$ for a.e. $t \in [0,T]$ (the subscript $D$ stands for ``dual''), and
$$
\left\{
\begin{array}{rcll}
\partial_t m + \mathrm{div}( w) &=& 0&\rm{in} \; (0,T)\times \T^d\\
m(0,\cdot) &=& m_0& \rm{in} \; \T^d.
\end{array}
\right.
$$
in the sense of distributions. Because of the integrability assumption on $w$, it follows that $t \mapsto m(t)$ has a unique narrowly continuous representative (cf. \cite{ags}). It is to this representative that we refer when we write $m(t)$, and thus $m(t)$ is well-defined as a probability density for all $t \in [0,T]$.

Define the functional
\begin{equation}\label{eq:dual}
\s{B}(m,w) = \int_{\bb{T}^d}  g(x)m(T,x)\dd x + \int_0^T \int_{\bb{T}^d} m(t,x)L\left(x,\frac{w(t,x)}{m(t,x)}\right) + F(x,m(t,x))\, \dd x \dd t
\end{equation}
on $\s{K}_D$. Recall that $L$ is defined just after \eqref{eq:hamiltonian_conjugate_bounds}.  We follow the convention that
\begin{equation}\label{conventionH*}
mL\left(x,\frac{w}{m}\right) = 
\left\{ 
\begin{array}{ll} 
+\infty, & {\rm{if}}\;\;\; m=0\ {\rm{and}}\ w\neq 0, \\ 
0, & {\rm{if}}\;\;\; m=0\ {\rm{and}}\ w = 0. 
\end{array} \right.
\end{equation}
Since $m \geq 0$, the second integral in (\ref{eq:dual}) is well-defined in $(-\infty,\infty]$ by the assumptions on $F$ and $L$. The first integral is well-defined and necessarily finite by the continuity of $ g$ and the fact that $m(T,x)\dd x$ is a probability measure.

We next state the ``dual problem" as \medskip
\begin{problem}[Dual Problem] \label{pr:dual}
Find $\ds\inf_{(m,w) \in \s{K}_D} \s{B}(m,w)$.
\end{problem} \medskip
\begin{proposition}\label{prop:duality}
Problems \ref{pr:smooth} and \ref{pr:dual} are in duality, i.e.
\begin{equation} \label{eq:duality}
\inf_{ u \in \s{K}_P} \s{A}( u) = -\min_{(m,w) \in \s{K}_D} \s{B}(m,w)
\end{equation}
Moreover, the minimum on the right-hand side is achieved by a pair $(m,w) \in \s{K}_D$ with $m\in L^\infty([0,T]\times\T^d)$ and $w\in L^{r'}([0,T]\times\T^d;\R^d).$
\end{proposition}

\begin{proof}
The proof relies on the Fenchel-Rockafellar  duality theorem (see for example \cite{EkTe}) and basically follows the lines of the proof of Lemma 2.1 from \cite{carda}, hence we omit it. 
The integrability of $(m,w)$ is just coming from the density constraint and from the growth condition of $H^*.$
\end{proof}

\begin{remark}\label{uniquemw} If $f$ is strictly increasing with respect to the second variable in $\T^d\times (0,\ov m)$, then the minimizer   $(m,w)$ is unique. 
\end{remark}

\medskip
In general one cannot expect Problem \ref{pr:smooth} to have a solution. This motivates us to relax it and search for solutions in a larger class. For this let us first state the following observation.
\begin{lemma}\label{lem:esti} Let $( u_n)$ be a minimizing sequence for Problem \ref{pr:smooth} and set $\alpha_n = -\partial_t  u_n + H(x,D u_n)$. Then $( u_n)$ is bounded in $BV([0,T]\times\T^d)\cap L^r([0,T]\times\T^d)$, the sequence $(\alpha_n)$ is bounded in $L^1([0,T]\times\T^d)$, with $\alpha_n\geq 0$ a.e., while $(D u_n)$ is bounded in $L^r([0,T]\times\T^d)$. Finally, there exists a Lipschitz continuous function $\psi:[0,T]\times \T^d\to \R$ such that $\psi(T,\cdot)= g$ and  $u_n\geq \psi$ for any $n$. 
\end{lemma}
\begin{proof}
As $F^*(\cdot,\alpha)=0$ for $\alpha\leq 0$, we can assume without  loss of generality that  $\alpha_n\geq 0$ (indeed, if we replace $\alpha_n$ with its positive part, the part with $F^*$ does not increase, and the value of $u_n(0,x)$ does not decrease, by the maximum principle applied to the HJ equation connecting $u_n$ to $\alpha_n$). By comparison, $ u_n\geq \psi$ where $\psi$ is the unique Lipschitz continuous viscosity  solution to 
$$
\left\{
\begin{array}{rcll}
-\partial_t \psi + H(x,D\psi) & = & 0& \rm{in}\; (0,T)\times \T^d\\
\psi(T,x) &=&  g(x)& \rm{in}\;  \T^d.
\end{array}
\right.
$$
So $( u_n)$ is uniformly bounded from below. Integrating the equation for $( u_n)$ on $[0,T]\times \T^d$ and using the fact that $H\geq -C$ and the fact that $ g$ is bounded, we get  (up to redefining the constant $C>0$)
$$
\int_{\T^d}  u_n(0,x)\dd x \leq \int_0^T\int_{\T^d} \alpha_n\dd x\dd t + C.
$$
So, by \eqref{remF*} and for $n$ large enough,  
$$
\begin{array}{rl}
\inf_{ u \in \s{K}_P} \s{A}( u)+1 \; \geq & \ds \int_0^T \int_{\T^d} F^*(x,\alpha_n)\dd x\dd t - \int_{\T^d}  u_n(0,x)m_0\dd x \\
\geq & \ds  \int_0^T  \int_{\T^d}   \ov m \alpha_n\dd x\dd t - \int_{\T^d}  u_n(0,x)m_0\dd x -C\\
\geq & \ds \int_{\T^d}  u_n(0,x)(\ov m- m_0)\dd x -C.
\end{array}
$$
By {\bf (H2)} $\ov m- m_0>\ov c$: as we know that $ u_n(0,\cdot)$ is bounded from below, we get that $( u_n(0,\cdot))$ is bounded in $L^1(\T^d)$. Thus, as $\alpha_n\geq 0$, we also have that $(\alpha_n)$ is bounded in $L^1([0,T]\times \T^d)$. Then integrating the equation $\alpha_n = -\partial_t  u_n + H(x,D u_n)$ over $[t,T]\times \T^d$ and using the lower bound on $H$, we get on the one hand  
$$
\int_{\T^d}  u_n(t,x)\dd x \leq \int_t^T\int_{\T^d} \alpha_n\dd x\dd t  + C,
$$
which, in view of the lower bound on $ u_n$, gives an $L^\infty([0,T])$ bound on $\ds\lg  u_n(t,\cdot)\rg= \int_{\T^d}  u_n(t,x)\dd x$. We  integrate again the equation $\alpha_n = -\partial_t  u_n + H(x,D u_n)$ over $[0,T]\times \T^d$ and use the coercivity of $H$ and Poincar\'e's inequality to get 
\begin{align*}
\ds C\; \geq &  \int_0^T\int_{\T^d} \alpha_n\dd x\dd t  + C \; \geq \; \ds  \int_0^T\int_{\T^d} H(x,D u_n)\dd x\dd t \; 
\geq \; \ds (1/C) \int_0^T\int_{\T^d}  |D u_n|^r\dd x\dd t - C\\ 
\geq & \ds (1/C) \int_0^T\int_{\T^d}  | u_n- \lg  u_n(t,\cdot)\rg|^r\dd x\dd t -C \; 
\geq \; \ds (1/C) \int_0^T\int_{\T^d}  | u_n|^r\dd x\dd t -C.
\end{align*} 
In particular $(D u_n)$ and $(u_n)$ are bounded in $L^r([0,T]\times\T^d)$. Thus $\partial_t  u_n = - \alpha_n+ H(x,D u_n)$ is bounded in $L^1([0,T]\times\T^d)$. The result follows. 
\end{proof}

By the results of Lemma \ref{lem:esti} we introduce a relaxation of the Problem \ref{pr:smooth}. 
Let us denote by ${\mathcal K}_{R}$ (here the subscript $R$ stands for ``relaxed'') the set of pairs $( u, \alpha)$ such that $ u\in BV([0,T]\times\T^d)$ with $D u\in L^r([0,T]\times\T^d;\R^d)$ and $u(T^-,\cdot)\geq g$ a.e., $\alpha$ is a nonnegative measure on $[0,T]\times \T^d$, and, if we extend $(u,\alpha)$ by setting $u= g$ and $\a:= H(\cdot, Dg)\dd x\dd t$ on $(T, T+1)\times \T^d$, then the pair $(u,\a)$ satisfies 
$$
-\partial_t  u+ H(x,D u)\leq \alpha
$$
in the sense of distribution in $(0,T+1)\times \T^d$. Note that the extension of $(u,\a)$ to $[0,T+1]\times \T^d$ just expresses the fact that $u(T^+)= g$ and that $\alpha$ compensates the possible jump from $u(T^-)$ to $g$. We set
$$
\s{A}( u,\alpha) = \int_0^T \int_{\bb{T}^d} F^*(x,\alpha^{ac}(t,x)) \dd x \dd t + \ov m \a^s([0,T]\times \T^d) - \int_{\bb{T}^d}  u(0^+,x)\, \dd m_0(x) .
$$
where $\alpha^{ac}$ and $\alpha^s$ are respectively  the absolutely continuous part and the singular part of the measure $\alpha$. 
\begin{problem}[Relaxed Problem] \label{pr:relaxed}
Find $\ds\inf_{( u,\alpha) \in \s{K}_{R}} \s{A}( u,\alpha)$.
\end{problem} \medskip
Let us consider the following result as a counterpart of Lemma 2.7 from \cite{cargra} in our case.

\begin{lemma}\label{lem:ineqmwphialpha} Let $(m, w)\in {\mathcal K}_D$ such that $m\in L^\infty([0,T]\times\T^d)$ and $( u, \alpha)\in {\mathcal K}_R$ an arbitrary competitor for Problem \ref{pr:relaxed}. Then, for every $t\in [0,T]$, we have  
\begin{align*}
\int_{0}^{t}\int_{\T^d} - m H^*\left(x,-\frac{w}{m}\right)\dd x\dd t &\leq\int_{\T^d} m(t,x)  u(t^-,x)\dd x- \int_{\T^d} m_0(x)  u(0^+,x)\dd x\\
&\qquad + \int_{0}^{t}\int_{\T^d} \alpha^{ac} m\dd x\dd t + \ov m \a^s([0, t]\times \T^d) 
\end{align*}
and 
\begin{align*}
\int_t^{T}\int_{\T^d} - m H^*\left(x,-\frac{w}{m}\right)\dd x\dd t &\leq\int_{\T^d} m(T,x)  g(x)\dd x- \int_{\T^d} m(t,x)  u(t^+,x)\dd x\\
&\qquad + \int_t^{T}\int_{\T^d} \alpha^{ac} m\dd x\dd t + \ov m \a^s([t, T]\times \T^d). 
\end{align*}
Moreover we can take $t=0$ in the above inequalities. 
If, finally, equality holds in the second inequality when $t=0$, then $w= -mD_pH(\cdot, Du)$ a.e. and 
$$
\limsup_{\ep\to 0} m_\ep(t,x)= \ov m \qquad for \; \a^s-a.e.\ (t,x)\in [0,T]\times \T^d,
$$
where $m_\ep$ is any standard mollification of $m$. 
\end{lemma}
\begin{proof} We prove the result only for $t=0$, the general case follows by a similar (and simpler) argument. 
We first extend the pairs $(u,\a)$ and $(m,w)$ to $(0,T+1)\times \T^d$ by setting $u= g$ and $\a:= H(\cdot, Dg)\dd x\dd t$, $m(s,x)=m(T,x)$, $w(s,x)=0$ on $(T,T+1)\times \T^d$. Note that  
$$\partial_t m+\diver(w)=0\;\;\; {\rm{and}}\;\;\; -\partial_t  u+ H(x,D u)\leq \alpha\;\;\; {\rm{on}}\;\; (0,T+1)\times \T^d.$$  
We smoothen the pair $(m,w)$ in a standard way into $(m_\e,w_\e)$: $m_\ep:= m\star \rho_\ep$ and $w_\ep:= w\star \rho_\ep$, where the mollifier $\rho$ has a support in the unit ball of $\R^{d+1}$ and $\rho_\ep:= \ep^{-d-1}\rho(\cdot/\ep)$. Then, for any $\eta>\e$, we have, since $m_\e\leq \ov m$,  
\be\label{gduoysd}
\begin{array}{l}
\ds \int_\eta^{T+\eta}\int_{\T^d}  u \partial_t m_\e +m_\e H(x,D u)\dd x\dd t - \left[ \int_{\T^d} m_\ep  u\dd x\right]_{\eta^+}^{T+\eta}\\
\ds \qquad \qquad  \leq \int_{\eta}^{T+\eta}\int_{\T^d} m_\e \dd \alpha
\leq \int_{\eta}^{T+\eta}\int_{\T^d} \alpha^{ac} m_\e\dd x\dd t
+ \ov m \alpha^{s}([\eta, T+\eta]\times \T^d) 
\end{array}
\ee
where, as $\partial_t m_\e+ \mathrm{div}( w_\e)=0$,  
$$
\int_\eta^{T+\eta}\int_{\T^d}  u \partial_t m_\e\dd x\dd t= \int_\eta^{T+\eta}\int_{\T^d} D u\cdot w_\e\dd x\dd t.
$$
In the followings, we shall consider only such $\eta$'s for which no jump of $u$ occurs, in particular $\ds\int_{\T^d}u(\eta^-,x)\dd x=\ds\int_{\T^d}u(\eta^+,x)\dd x=\ds\int_{\T^d}u(\eta,x)\dd x.$
So, by convexity of $H$,  \eqref{gduoysd} and the above equality, 
$$
\begin{array}{l}
\ds \int_\eta^{T+\eta}\int_{\T^d} - m_\e H^*\left(x,-\frac{w_\e}{m_\e}\right)\dd x\dd t \; 
\leq \; \ds \int_\eta^{T+\eta}\int_{\T^d} w_\e\cdot D u + m_\e H(x,D u)\dd x\dd t \\
\qquad \qquad \leq \; \ds  \left[ \int_{\T^d} m_\e  u\dd x\right]_{\eta}^{T+\eta} + \int_\eta^{T+\eta}\int_{\T^d} \alpha^{ac} m_\e\dd x\dd t
+ \ov m \alpha^{s}([\eta, T+\eta]\times \T^d) .
\end{array}
$$
We multiply the inequality $-\partial_t  u+ H(x,D u)\leq \alpha$ by $m_\ep(\eta,\cdot)$ and integrate on $(0,\eta)\times \T^d$ to get, as $m_\ep$ is bounded by $\ov m$ and $H$ is bounded from below,  
$$
\int_{\T^d}  u(0^+,x)m_\ep(\eta,x)\dd x \leq   \int_{\T^d}  u(\eta,x)m_\ep(\eta,x)\dd x + C\eta + \ov m \alpha((0,\eta)\times \T^d) .
$$
Note that $\ov m \alpha((0,\eta)\times \T^d)= \ov m \a^s([0,\eta)\times \T^d)+o(1)$, where $o(1)\to 0$ as $\eta\to 0$. 
So
$$
\begin{array}{l}
\ds \int_\eta^{T+\eta}\int_{\T^d} - m_\ep H^*\left(x,-\frac{w_\ep}{m_\ep}\right)\dd x\dd t \le\\ 
 \leq \; \ds  \int_{\T^d}  [m_\ep(T+\eta,x)  u(T+\eta,x)- u(0^+,x)m_\ep(\eta,x)]\dd x\\
 \qquad \qquad \ds + \int_\eta^{T+\eta}\int_{\T^d} \alpha^{ac} m_\ep\dd x\dd t
+ \ov m \a^s([0, T+\eta]\times \T^d) +o(1).
\end{array}
$$
We now let $\ep\to0$. 
By convergence of $(m_\ep, w_\ep)$ to $(m,w)$ in $L^q\times L^{r'},\;\forall q\ge1$, and by the fact that  
$$\lim_{\e\to 0}\int_\eta^{T+\eta}\int_{\T^d} m_\e H^*(x,-w_\e/m_\e)\dd x\dd t=\int_\eta^{T+\eta}\int_{\T^d} mH^*(x,-w/m)\dd x\dd t,$$ 
(see the proof of Lemma 2.7. from \cite{cargra}) we obtain for all $\eta>0$, chosen above,
$$
\begin{array}{l}
\ds \int_\eta^{T+\eta}\int_{\T^d} - m H^*\left(x,-\frac{w}{m}\right)\dd x\dd t \le\\
 \leq \; \ds   \int_{\T^d}  [m(T+\eta,x)  u(T+\eta,x)- u(0^+,x)m(\eta,x)]\dd x \\
  \qquad \qquad \ds+ \int_\eta^{T+\eta}\int_{\T^d} \alpha^{ac} m\dd x\dd t
+ \ov m \a^s([0, T+\eta]\times \T^d)+o(1). 
\end{array}
$$
By definition of the extension of the maps $ u$ and $m$, 
$$
\begin{array}{l}
\ds  \int_{\T^d}  [m(T+\eta)  u(T+\eta)- u(0^+,x)m(\eta,x)]\dd x+ \int_\eta^{T+\eta}\int_{\T^d} \alpha^{ac} m\dd x\dd t
+ \ov m \a^s([0, T+\eta]\times \T^d)\\
\qquad \ds = 
 \int_{\T^d}  [m(T,x)  g(x)- u(0^+,x)m(\eta,x)]\dd x+ \int_\eta^{T+\eta}\int_{\T^d} \alpha^{ac} m\dd x\dd t
+ \ov m \a^s([0, T]\times \T^d).
\end{array}
$$
We finally let $\eta\to 0$ and get 
$$
\begin{array}{l}
\ds \int_0^{T}\int_{\T^d} - m H^*\left(x,-\frac{w}{m}\right)\dd x\dd t \\
\qquad \ds \le 
 \int_{\T^d}  [m(T,x)  g(x)- u(0^+,x)m_0(x)]\dd x+ \int_0^{T}\int_{\T^d} \alpha^{ac} m\dd x\dd t
+ \ov m \a^s([0, T]\times \T^d)
\end{array}
$$
thanks to the $L^\infty-$weak-$\star$ continuity of $t\mapsto m(t)$ and the $L^1$ integrability of $u(0^+,\cdot)$. 

The proof of the equality $w=-mD_pH(\cdot, Du)$ when equality holds in the above inequality follows exactly the proof of the corresponding statement in \cite{cargra}, so we omit it. Note that, if equality holds, then all the above inequalities must become equalities as $\ep$ and then $\eta$ tend to $0$. In particular, from inequality \eqref{gduoysd}, we must have 
$$
\lim_{\eta\to 0} \limsup_{\ep\to 0} \int_\eta^{T+\eta} \int_{\T^d} m_\ep(t,x)\dd\a^s(t,x)= \ov m \a^s([0,T]\times \T^d). 
$$
By Fatou's lemma, this implies that 
$$
\ov m \a^s([0,T]\times \T^d)\leq \int_0^T \int_{\T^d} \limsup_{\ep\to 0} m_\ep(t,x)\dd\a^s(t,x), 
$$
where the right-hand side is also bounded above by the left-hand side since $m_\ep\leq \ov m$. So $\ds \limsup_{\ep\to 0} m_\ep=\ov m$ $\a^s-$a.e.
\end{proof}

\begin{proposition}\label{inf=min} We have 
\be\label{ineq:relax}
 \inf_{ u\in {\mathcal K}_P} \s{A}( u) = \min_{( u,\alpha)\in {\mathcal K}_R} \s{A}( u,\alpha) . 
\ee
Moreover, if $(u,\a)$ is a minimum of $\s{A}$, then $\a= \a\mres [0,T)\times \T^d + (u(T^-,\cdot)-g)\dd\left(\delta_T\otimes{\mathcal H}^d\mres\T^d\right)$. 
\end{proposition}

\begin{proof} We follow here \cite{graber}. Inequality $\geq$ is obvious and we now prove the reverse one. Let us fix $( u,\alpha)\in {\mathcal K}_R$ and let $(m,w)$ be an optimal solution for the dual problem. Then 
$$
\begin{array}{rl}
\ds \s{A}( u,\alpha) \; = & \ds \int_0^T\int_{\T^d} F^*(x,\alpha^{ac})\dd x\dd t+\ov m \a^s([0,T]\times \T^d) - \int_{\T^d} m_0(x) u(0^+,x)\dd x\\
\geq & \ds  \int_0^T\int_{\T^d} (m\alpha^{ac}- F(x,m))\dd x\dd t +\ov m\a^s([0,T]\times \T^d) - \int_{\T^d} m_0(x) u(0^+,x)\dd x\\
\geq & \ds \int_0^T\int_{\T^d}(-mH^*(x,-w/m) -F(x,m))\dd x\dd t - \int_{\T^d} m(T,x) g(x)\dd x \; = \; -{\mathcal B}(m,w)
\end{array}
$$
where the last inequality comes from Lemma \ref{lem:ineqmwphialpha}. By optimality of $(m,w)$ and \eqref{eq:duality} we obtain therefore
$$
\s{A}( u,\alpha)  \geq \inf_{ u\in {\mathcal K}_P} \s{A}( u),
$$
which shows equality \eqref{ineq:relax}.

To prove that the problem in the right-hand side of \eqref{ineq:relax} has a minimum we consider a minimizing sequence $( u_n)$ for Problem \ref{pr:smooth}. We extend $u_n= g$ on $(T,T+1]\times \T^d$ and set $\alpha_n:= -\partial_t  u_n + H(x,D u_n)$ on $[0,T+1]\times \T^d$ and note that, in view of Lemma \ref{lem:esti}, there is a  subsequence, again denoted by $( u_n,\alpha_n)$, such that $(u_n)$ converges in $L^1$ to a $BV$ map $ u$, $(D u_n)$ converges weakly in $L^r$, and $(\alpha_n)$  converges in sense of measures to $\alpha$ on $[0,T+1]\times \T^d$. As $u_n\geq \psi$ on $[0,T]\times \T^d$, we also have $u\geq \psi$ $(0,T)\times \T^d$, so that $u(T^-,\cdot)\geq \psi(T,\cdot)= g$. By convexity of $H$ with respect to $p$, the pair $( u,\alpha)$ belongs to ${\mathcal K}_R$. One easily shows by standard relaxation that 
$$
\s{A} ( u,\alpha) \leq \liminf_{n\to\infty} \s{A}( u_n).
$$
Hence $( u,\alpha)$ is a minimum. 

Let us finally check that $\a= \a\mres (0,T)\times \T^d + (u(T^-,\cdot)-g)\dd\left( \delta_T\otimes {\mathcal H}^d\mres\T^d\right)$. Indeed, by definition of ${\mathcal K}_R$, we can extend $(u,\alpha)$ by setting $(u,\a):= (g, H(\cdot,Dg))$ on $(T,T+1)\times \T^d$ and the following inequality holds in the sense of measure in $(0,T+1)\times \T^d$:  
$$
-\partial_t u + H(x,Du)\leq \alpha
$$
Let $\phi\in C^\infty(\T^d)$, with $\phi\geq 0$, be a test function.  We multiply the above inequality by $\phi$ and integrate on $(T-\eta,T+\eta)\times \T^d$ to get
$$
\int_{\T^d} \phi(x) u((T-\eta)^+,x)\dd x + \int_{T-\eta}^{T+\eta}\int_{\T^d}\phi H(x,Du)\dd x\dd t\leq \int_{\T^d} g\phi\dd x+  \int_{T-\eta}^{T+\eta}\int_{\T^d}\phi \dd\alpha. 
$$
Letting $\eta\to 0$ along a suitable sequence such that $u((T-\eta)^+,\cdot)\to u(T^-,\cdot)$ in $L^1$, we obtain: 
$$
\int_{\T^d} \phi u(T^-,x)\dd x \leq \int_{\T^d} g\phi\dd x+  \int_{\T^d}\phi \dd(\alpha\mres\{T\}\times \T^d). 
$$
This means that $u(T^-,\cdot)\leq g+ \alpha\mres\{T\}\times \T^d$. Let us now replace $\alpha$ by $$\tilde \alpha:= \alpha\mres (0,T)\times \T^d + (u(T^-,\cdot)-g)\dd\left(\delta_T\otimes {\mathcal H}^d\mres\T^d\right).$$ 
We claim that the pair $(u,\tilde \alpha)$ still belongs to ${\mathcal K}_R$. For this we just have to check that, if we extend $(u,\tilde \a)$ to $(0,T+1)\times \T^d$ as before, then 
$$
-\partial_t u + H(x,Du)\leq \tilde \alpha\; {\rm{on}} \; (0,T+1)\times \T^d 
$$
holds in the sense of distributions.
Let $\phi\in C^\infty_c((0,T+1)\times \T^d)$ with $\phi\geq 0$. Then
\begin{align*}
\ds\int_0^{T+1}\int_{\T^d}& u\partial_t \phi + \phi H(x,Du)\dd x\dd t= \\ 
&=  \int_0^T \int_{\T^d} u\partial_t \phi + \phi H(x,Du)\dd x\dd t
+ \int_T^{T+1}\int_{\T^d} g \partial_t \phi+  \phi H(x,Dg)\dd x\dd t \\
& \leq  \int_0^T \int_{\T^d} \phi\dd(\a \mres (0,T)\times \T^d)  + \int_{\T^d} (u(T^-,x)-g) \phi(T,x)\dd x+ \int_T^{T+1}\int_{\T^d}
\phi H(x,Dg)\dd x\dd t\\
& \leq \ds  \ds \int_0^{T+1} \int_{\T^d} \phi\dd\tilde \a .
\end{align*}
This proves that  the pair $(u,\tilde \alpha)$ belongs to ${\mathcal K}_R$. In particular, ${\mathcal A}(u,\alpha)\leq {\mathcal A}(u,\tilde \a)$, so that
$$
\ov m \a^s(\{T\}\times \T^d) \leq \ov m  \int_{\T^d} u(T^-,x)-g(x)\dd x.
$$
Since we have proved that $u(T^-,\cdot)\leq g+ \alpha\mres\{T\}\times \T^d$, we have therefore an equality in the above inequality, which means that $\alpha\mres\{T\}\times \T^d= (u(T^-,\cdot)- g)\dd x$.
\end{proof}

\section{The MFG system with density constraints}\label{sec:MFGdensity}

In this section, we study the existence of solutions for the MFG system with density constraints: 
\be\label{MFGdensity}
\left\{
\begin{array}{crcll}
{\rm{(i)}} & -\partial_t u + H(x,D u) &=& f(x,m)+\beta & {\rm{in}}\; (0,T)\times \T^d \\
{\rm{(ii)}}& \partial_t m - \mathrm{div} \left(mD_p H(x,D u)\right) &=& 0  &{\rm{in}}\; (0,T)\times \T^d\\

{\rm{(iii)}} & u(T,x) =  g(x)+\beta_T,&\beta_T\geq 0,&\beta_T(\overline m-m)=0& {\rm{in}}\; \T^d\\

{\rm{(iv)}} & 0\leq m, \leq \overline m & \beta\geq 0, & \beta(\overline m-m)=0& {\rm{in}}\; (0,T)\times \T^d\\

{\rm{(v)}} &  &m(0,x) = m_0(x)  &\;& {\rm{in}}\; \T^d\\

\end{array}\right.
\ee
under the assumptions on $H$, $f$, $g$ and $m_0$ stated in Section \ref{sec:Hyp}.  We also study the approximation of the solution of this system by the solution of the classical MFG system. 

\subsection{Solutions of the MFG system with density constraints}

\begin{definition}\label{def:solution} 
We say that $( u, m, \beta,\beta_T)$ is a solution to the MFG system \eqref{MFGdensity} if 
\begin{enumerate}
\item Integrability conditions:   $\beta$ is a nonnegative Radon measure on $(0,T)\times \T^d$, $\beta_T\in L^1(\T^d)$ is nonnegative,  
$ u\in BV([0,T]\times\T^d)\cap L^r([0,T]\times\T^d)$, $D u\in L^r([0,T]\times\T^d;\R^d)$, $m\in L^1([0,T]\times\T^d)$ and $0\leq m \leq \ov m$ a.e., 

\item The following inequality 
\be\label{IneqHJ}
-\partial_t  u+H(x,D u(t,x))\leq f(x,m)+\beta
\ee
holds in $(0,T)\times \T^d$ in the sense of measures,
with the boundary condition on $\T^d$
$$
g\leq u(T^-,\cdot)= g+\beta_T \; {\rm{a.e.}}
$$
Moreover,   $\beta^{ac}= 0$ a.e. in $\{m<\ov m\}$ and
\be\label{m=ovmaspp}
\limsup_{\ep\to 0} m_\ep(t,x)= \ov m\qquad \beta^s-a.e.\; {\rm{if}}\;t<T \; {\rm{and}}\; {\rm{a.e.}} \; {\rm{in}}\;  \{\beta_T>0\}\; {\rm{if}} \; t=T,
\ee
where $m_\ep$ is any standard mollification of $m$. 

\item Equality 
$$
\partial_t m-\diver(mD_pH(x,D u(t,x))= 0, \qquad m(0)=m_0
$$
holds in the sense of distribution, 

\item\label{point4} Equality 
\begin{align*}
\int_0^T\int_{\T^d} m(-H(x,D u) & +Du\cdot D_pH(x,D u) +f(x,m)+ \beta^{ac})\dd x\dd t \\ 
+ \ov m \beta^s((0,T)\times \T^d)+ \ov m \int_{\T^d}\beta_T\dd x& = \int_{\T^d} m_0(x)u(0^+,x)- m(T,x)g(x)\dd x
\end{align*}
holds. 
\end{enumerate}
\end{definition}
Let us recall that, in the above definition, $\beta^{ac}$ and $\beta^s$ denote the absolutely continuous part and the singular part of the measure $\beta$. 

Some comments on the definition are now necessary. Equality \eqref{m=ovmaspp} is a weak way of stating that $m=\ov m$ in the support of $\beta$ and $\beta_T$ respectively, while the last requirement formally says that equality $-\partial_t  u+H(x,D u(t,x))= f(x,m)+\beta$ holds.

We can state the main result of this section. 
\begin{theorem}\label{thm:sol} 
Let $( u, \alpha)\in \s{K}_R$ be a solution of the relaxed Problem \ref{pr:relaxed} and $(m,w)\in \s{K}_D$ be a solution of the dual Problem \ref{pr:dual}. Then $\alpha\geq f(\cdot,m)$ as measures, and, if we set 
$$\beta:= \alpha \mres [0,T)\times \T^d-f(\cdot, m) \dd x\dd t 
$$
 and $\beta_T:= \alpha \mres \{T\}\times \T^d$, 
the quadruplet $( u,m,\beta,\beta_T)$ is a solution of the MFG system \eqref{MFGdensity}. 

Conversely, let $( u,m,\beta,\beta_T)$ be a solution of the MFG system \eqref{MFGdensity}. Let us set 
\be\label{def.alpha}
\alpha:= f(\cdot, m) \dd x\dd t+\beta+\beta_T \dd(\d_T\otimes {\mathcal H}^{d}\mres \T^d)
\ee and
$w= -mD_pH(x, Du)$. Then the pair $( u, \alpha)$ is a solution of the relaxed problem while the pair $(m,w)$ is a solution of the dual problem. 
\end{theorem}

The proof of this results goes along the same lines as in \cite{cargra} (Theorem 3.5). However for the sake of completeness (and because of some differences) we sketch it here.
\begin{proof} 
Let $(u,\a)\in\cK_R$ be a solution of the Problem \ref{pr:relaxed} and $(m,w)\in \cK_D$ be the solution of Problem \ref{pr:dual}. First, by the definition of Legendre-Fenchel transform we have for a.e. $(t,x)\in [0,T]\times \T^d$
\begin{equation}\label{lf:alpha}
F^*(x,\a^{ac}(t,x)) +F(x,m(t,x))-\a^{ac}(t,x)m(t,x)\ge 0.
\end{equation}
On the other hand by optimality we have that 
\begin{align*}
0&=\cA(u,\a)+\cB(m,w)\\
& = \int_0^T \int_{\bb{T}^d} F^*(x,\alpha^{ac}(t,x)) \dd x \dd t + \ov m \a^s([0,T]\times \T^d) - \int_{\bb{T}^d}  u(0^+,x)\, m_0(x)\dd x\\
& + \int_{\bb{T}^d}  g(x)m(T,x)\dd x + \int_0^T \int_{\bb{T}^d} m(t,x)H^*\left(x,-\frac{w(t,x)}{m(t,x)}\right) + F(x,m(t,x))\, \dd x \dd t\\
&\ge \int_0^T\int_{\T^d}\a^{ac}m\dd x\dd t + \ov m \a^s([0,T]\times \T^d) + \int_{\bb{T}^d}g(x)m(T,x) -  u(0^+,x)\, m_0(x)\dd x\\
& + \int_0^T \int_{\bb{T}^d} m(t,x)H^*\left(x,-\frac{w(t,x)}{m(t,x)}\right)\dd x\dd t\; \ge \;0,
\end{align*}
where we used the Lemma \ref{lem:ineqmwphialpha} for the last inequality. This means that all the inequalities in the previous lines are equalities. In particular we have an equality in \eqref{lf:alpha}, which implies 
$$\a^{ac}(t,x)\in\partial_m F(x,m(t,x)) \qquad {\rm{a.e.}}$$
As $\partial_m F(x,m(t,x))= \{f(x,m(t,x))\}$ for $0<m(t,x)<\ov m$ a.e., we have $\a^{ac}(t,x)= f(x,m(t,x))$ a.e. in $\{0<m< \ov m\}$. Moreover, as $\partial_m F(x,0)=(-\infty,0]$ and $\alpha^{ac}\geq 0$, we also have $\a^{ac}=0= f(\cdot,0)$ a.e. in $\{m=0\}$. Finally, since   $\partial_m F(x,\ov m)=[f(x,\ov m),+\infty)$, $\a^{ac}\geq f(x,m(t,x))$ a.e. on $\{m=\ov m\}$. Therefore $\a^{ac}\geq f(\cdot,m)$ a.e. Let us set $\beta := \alpha\mres [0,T)\times \T^d-f(\cdot, m)\dd x\dd t$ and $\beta_T:=\alpha\mres \{T\}\times \T^d$. From Proposition \ref{inf=min} we know that $\beta_T= u(T^-,\cdot)-g$. 

Since equality holds in the above inequalities, there is an equality in the inequality of Lemma \ref{lem:ineqmwphialpha}: thus  point \ref{point4} holds in Definition \ref{def:solution}. Moreover, by Lemma \ref{lem:ineqmwphialpha}, we have that $w= -m D_pH(\cdot, Du)$ a.e. and \eqref{m=ovmaspp}  holds. In conclusion, since $(u,\a)\in\cK_R$ and $(m,w)\in \cK_D$, the quadruplet $(u,m,\beta,\beta_T)$ is a solution to the MFG system \eqref{MFGdensity}.

Now let us prove the converse statement. For this let us take a solution $(m,u,\beta,\beta_T)$ of the MFG system \eqref{MFGdensity} in the sense of the Definition \ref{def:solution}. Let us define $\alpha$ as in \eqref{def.alpha} and $w:= -m D_pH(\cdot,Du)$.  We shall prove that $(u,\a)$ is a solution for the Problem \ref{pr:relaxed} and $(m,w)$ is a solution of Problem \ref{pr:dual}. For the first one, one easily checks, following the argument of Proposition \ref{inf=min}, that $(u,\a)\in {\mathcal K}_R$. Let us now consider a competitor $(\tilde u,\tilde\a)\in\cK_R.$ Using the equality in Lemma \ref{lem:ineqmwphialpha} for $(u,\a,m,-mD_pH(\cdot,Du))$ and the inequality for $(\tilde u,\tilde \a,m,-mD_pH(\cdot,Du))$  we have
\begin{align*}
\cA(\tilde u,\tilde\a)&=\int_0^T\int_{\T^d} F^*(x,\tilde\a^{ac}(x))\dd x\dd t +\ov m\tilde\a^s([0,T]\times\T^d)-\int_{\T^d}\tilde u(0^+,x)m_0(x)\dd x\\
&\ge \int_0^T\int_{\T^d} F^*(x,\a^{ac}(x)) + m(\tilde\a^{ac}-\a^{ac})\dd x\dd t +\ov m\tilde\a^s([0,T]\times\T^d)-\int_{\T^d}\tilde u(0^+,x)m_0(x)\dd x\\
&\ge\int_0^T\int_{\T^d} F^*(x,\a^{ac}(x))\dd x\dd t + \ov m\a^s([0,T]\times\T^d) - \int_{\T^d} u(0^+,x)m_0(x)\dd x,
\end{align*}
thus $(u,\a)$ is a minimizer for the Problem \ref{pr:relaxed}.
 
In a similar manner we can show that $(m,w)$ is a solution for Problem \ref{pr:dual}. Hence the statement of the theorem follows.
\end{proof}

We now briefly discuss the issue of the approximation of any solution to the MFG system with density constraints. If $F=F(x,m)$ is strictly convex on $[0,\ov m]$ with respect to the $m$ variable, then, as $H^*=H^*(x,q)$ is strictly convex with respect to $q$ (because $H=H(x,p)$ is $C^1$ in $p$), we can conclude that the dual Problem \ref{pr:dual} has a unique minimizer. In particular, in this case, the $m$ component of the MFG system \eqref{MFGdensity} is unique. We do not expect uniqueness of the $u$ component: this is not the case in the ``classical setting", i.e., without density constraint (see, however, the discussion in \cite{carda}). For this reason, the fact that one can approximate {\it any} solution of  the MFG system \eqref{MFGdensity} by regular maps with suitable property is not straightforward. This is the aim of the next lemma, needed in the sequel, where we explain that the $\beta$ component of any solution can be approached by a minimizing sequence of Lipschitz maps with some optimality property. 

\begin{lemma}\label{lem.approx} Let $(u, m, \beta,\beta_T)$ be a solution to the MFG system \eqref{MFGdensity}. Then there exist Lipschitz continuous maps $(u_n, \alpha_n)$ such that 
\begin{itemize}
\item[(i)] $u_n$ satisfies a.e. and in the viscosity sense,
$$
-\partial_t u_n +H(x,Du_n) = \alpha_n \; \rm{in}\; (0,T)\times \T^d,
$$
\item[(ii)] the pair $(u_n,\alpha_n)$ is a minimizing sequence for Problem \ref{pr:smooth} and Problem \ref{pr:relaxed}, 
\item[(iii)] $(u_n)$ is bounded from below, is bounded in $BV([0,T]\times\T^d)\cap L^r([0,T]\times\T^d)$ and $(D u_n)$ is bounded in $L^r([0,T]\times\T^d)$, 
\item[(iv)] $(u_n)$ converges to some $\tilde u\geq u$ in $L^1((0,T)\times \T^d)$ with $\tilde u=u$ $m-$a.e.,
\item[(v)]  $(\alpha_n)$ is bounded in $L^1((0,T)\times \T^d)$ and converges in measure on $[0,T]\times \T^d$ to $\alpha$ defined from $(\beta,\beta_T)$ by \eqref{def.alpha},
\item[(vi)]  $(\tilde u, m,\beta,\beta_T)$ is a solution to the MFG system \eqref{MFGdensity}, 
\end{itemize}
\end{lemma}

\begin{proof} Let us define $\alpha$ as in \eqref{def.alpha} and recall that, by Theorem \ref{thm:sol},  $(u,\a)$ is a minimum in the relaxed Problem \ref{pr:relaxed}. In particular, $(u,\a)$ belongs to ${\mathcal K}_R$ , which means that, if we extend $(u,\a)$ by $(g,H(\cdot, Dg(\cdot)))$ in $[T,T+1]\times \T^d$, then 
$$
-\partial_t u +H(x,Du)\leq \alpha \; \rm{in}\; (0,T+1)\times \T^d.
$$
For $\eta\in (0,1)$ we set $u^\eta(t,x):=u(t+\eta, x)$ and $\a^\eta:=\tau_\eta\sharp \a$, where $\tau_\eta:[0,T+1)\times \T^d\to [-\eta, T+1-\eta)\times \T^d$ is the time shift $\tau_\eta(t,x)= (t-\eta, x)$. We then smoothen $u^\eta$ into  $u^\eta\star \rho_\ep$ where $\ep\in(0,\eta/2)$, $\rho$ is a standard even  mollifier supported in the unit ball of $\R^{d+1}$ and $\rho_\ep(\cdot):= \ep^{-d-1}\rho(\cdot/\ep)$. We note that $u^\eta\star \rho_\ep(t,x)$ for $t\in [T-\ep,T]$ is a mollified version of $g$. We finally slightly modify $u^\eta\star \rho_\ep$ so that it satisfies the boundary condition: let $\zeta:\R\to[0,1]$ be smooth, increasing, with $\zeta(s)=0$ for $s\leq -1$ and $\zeta(s)=1$ for $s\geq 0$. Set $\zeta_\ep(s)=\zeta(\ep^{-1}s)$, 
$u^{\eta,\ep}(t,x):= (1-\zeta_\ep(t-T))u^\eta\star \rho_\ep(t,x)+\zeta_\ep(t-T)g(x)$. 
Then $u^{\eta,\ep}(T,x)=g(x)$ and
$$
-\partial_t u^{\eta,\ep} +H(x,Du^{\eta,\ep})\leq \alpha^{\eta,\ep} \; \rm{in}\; (0,T+1)\times \T^d,
$$
where 
$$
\alpha^{\eta,\ep}:= \left[(1-\zeta_\e(t-T))\left(\a^\eta- H(\cdot, Du^\eta)\right)\star\rho_\ep+ H(x,Du^{\eta,\ep})-\zeta_\ep'(t-T)(g(x)-g\star \rho_\ep(t,x))\right]_+
$$
(observe that the last convolution is done in $(d+1)$ variables, even if $g$ is a function only depending on $x$). As $\ep\to0$, $u^{\eta,\ep}$ is bounded in BV and converges to $u^\eta$ in $L^1$ while $\a^{\eta,\ep}$ is nonnegative, bounded in $L^1$ and converges to $\a^\eta$ as a measure. We have 
$$
\begin{array}{rl}
\ds {\mathcal A}(u^{\eta,\ep},\a^{\eta,\ep}) \;  = & 
\ds \int_0^T\int_{\T^d} F^*\left(x,\a^{\eta,\ep}(t,x) \right)\dd x\dd t
-\int_{\T^d} u\star \rho_\ep(\eta,x)m_0(x) \dd x. 
\end{array}
$$
As $\ep\to0$, the first integral in the right-hand side converges to
\begin{multline*}
\int_0^T \int_{\bb{T}^d} F^*(x,(\alpha^{\eta})^{ac}(t,x)) \dd x \dd t + \ov m (\alpha^{\eta})^{s}((0,T]\times \T^d) \\
= \int_0^T \int_{\bb{T}^d} F^*(x,\alpha^{ac}(t+\eta,x)) \dd x \dd t + \ov m \a^s([\eta,T+\eta]\times \T^d).
\end{multline*}
This convergence is technical, but not difficult: without explicit dependence on $x$ this is just the combination of the l.s.c. behavior of this integral functional with the fact that convex functionals which are invariant by translation decrease by convolution; in the $x$-dependent case, one just needs to estimate the error using the regularity in $x$.

Pick now a sequence $(\eta_n)$ tending to 0, such that $u\star \rho_\ep(\eta_n,\cdot)$ converges in $L^1$ to $u(\eta_n,\cdot)$ as $\ep\to 0$ (this is the case for a.e. $\eta$) and $(u(\eta_n,\cdot))$ tends in $L^1$ to $u(0^+,\cdot)$ as $n\to+\infty$: then 
$$
\limsup_n \limsup_{\ep\to0} {\mathcal A}(u^{\eta_n,\ep}) \leq \limsup_n {\mathcal A}(u^{\eta_n}, \a^{\eta_n})= {\mathcal A}(u,\alpha).
$$
As $(u,\a)$ is a minimum in the relaxed Problem \ref{pr:relaxed}, we can find $\ep_n\to 0$ such that $(u^{\eta_n,\ep_n},\a^{\eta_n,\ep_n})$ is a minimizing sequence for Problem \ref{pr:relaxed} thanks to Proposition \ref{inf=min}. 

Let now $\tilde u_n$ be the viscosity solution to 
$$
\left\{\begin{array}{l}
-\partial_t u +H(x,Du)= \alpha^{\ep_n,\eta_n} \; \rm{in}\; (0,T)\times \T^d\\
u(T,x)= g(x)  \; \rm{in}\;  \T^d.
\end{array}\right.
$$
Standard results on viscosity solutions imply that $\tilde u_n$ is Lipschitz continuous  (because so are $\alpha^{\ep_n,\eta_n}$ and $g$), satisfies the equation a.e. and, by comparison, is such that  $\tilde u_n\geq u^{\eta_n,\ep_n}$. Therefore 
$$
{\mathcal A}(\tilde u_n,\alpha^{\ep_n,\eta_n})\leq  {\mathcal A}(u^{\eta_n,\ep_n},\a^{\eta_n,\ep_n}),
$$
so that $(\tilde u_n,\alpha^{\ep_n,\eta_n})$ is also a minimizing sequence  for Problem \ref{pr:relaxed}. By Lemma \ref{lem:esti}, $(\tilde u_n)$ is bounded from below, is bounded in $BV([0,T]\times\T^d)\cap L^r([0,T]\times\T^d)$ and $(D u_n)$ is bounded in $L^r([0,T]\times\T^d)$. Up to a subsequence, $(\tilde u_n)$ converges to a BV map $\tilde u$ in $L^1$ such that $\tilde u\geq u$. Note that, as in the proof of Proposition \ref{inf=min}, $(\tilde u, \alpha)$ is also a minimizer of Problem \ref{pr:relaxed}, so that, by Theorem \ref{thm:sol} , $(\tilde u, m,\beta,\beta_T)$ is also a solution to the MFG system \eqref{MFGdensity}. 
In particular, by (4) in the definition of solution, the inequalities of Lemma \ref{lem:ineqmwphialpha} must be equalities for $(\tilde u, \alpha)$ and $(u,\alpha)$ so that, 
$$
\int_{\T^d} m(t,x)u(t,x)\dd x =\int_{\T^d} m(t,x)\tilde u(t,x)\dd x \qquad \mbox{for a.e. $t\in [0,T]$.}
$$
 As $\tilde u\geq u$, this implies that $\tilde u=u$ $m-$a.e. In conclusion the pair $(\tilde u_n, \a^{\eta_n,\ep_n})$ satisfies our requirements. 
\end{proof}

\begin{remark}[About the uniqueness of $\beta$ in the solution of \eqref{MFGdensity}] Assuming that $f$ is increasing in its second variable,  we already know the uniqueness of $(m,w)$ (see Remark \ref{uniquemw}). This also gives the uniqueness of the density of $w$ w.r.t. $m$, i.e. of $D_pH(x,Du)$, on $\{m>0\}$. Supposing $H$ strictly onex, this also gives uniqueness of $Du$ a.e. on $\{m>0\}$. But if we formally differentiate equation (i) in \eqref{MFGdensity} we obtain only terms depending on $Du$, hence also $D\beta$ is unique on $\{m>0\}$. Using the BV regularity result of section 6 and the fact that $\beta$ vanishes on the non-negligible set $\{m<\overline m\}$ we infer uniqueness of $\beta$.  \end{remark}

\subsection{Approximation by classical MFG systems}

We now study to what extent the solution of the MFG system with density constraint introduced above can be obtained as the limit of the solutions of classical MFG systems. For this, we assume that $f^\ep:\T^d\times [0,+\infty)\to \R$ is a continuous function for each $\ep>0$, strictly increasing with respect to $m$, with $f^\ep(\cdot,0)=0$, and which fulfills the growth condition: there exists $\theta>1+d/r$ and $C, \ C_\ep>0$ such that 
$$
C^{-1} m^{\theta-1} -C \leq f^\ep(x,m)\leq C_\ep m^{\theta-1}+C_\ep.
$$
We consider $(u^\ep,m^\ep)$ the solution to the classical MFG system
\be\label{MFG1}
\left\{
\begin{array}{crcll}
{\rm{(i)}} & -\partial_t u^\ep + H(x,D u^\ep) &=& f^\ep(x,m^\ep) &  {\rm{in}}\; (0,T)\times \T^d \\
{\rm{(ii)}}& \partial_t m^\ep - \mathrm{div} \left(m^\ep D_p H(x,D u^\ep)\right) &=& 0  &{\rm{in}}\; (0,T)\times \T^d\\
{\rm{(iii)}} & u^\ep(T,x) =  g(x),&\;&  m^\ep(0,x) = m_0(x)  & {\rm{in}}\; \T^d\\
\end{array}\right.
\ee
 Following Cardaliaguet \cite{carda}, Cardaliaguet-Graber \cite{cargra}, we know that the MFG system \eqref{MFG1} has a unique (weak) solution $(u^\e,m^\ep)$:  namely, $(u^\ep,m^\ep)\in C^0([0,T]\times\T^d) \times L^\theta([0,T]\times\T^d)$ and the following hold: 
\begin{itemize}
\item[(i)] the following integrability conditions hold:
$$\ds Du^\ep\in L^r, \; \ds m^\ep H^*(\cdot, D_pH(\cdot,Du^\ep))\in L^1 
\quad{\rm{and}}\quad m^\ep D_pH(\cdot,Du^\ep))\in L^1 .
$$ 

\item[(ii)] Equation \eqref{MFG1}-(i) holds in the following sense: the inequality
\be\label{eq:distrib}
\ds \quad -\partial_t u^\ep+H(x,Du^\ep)\leq  f(x,m^\ep) \quad {\rm{in}}\; (0,T)\times \T^d, 
\ee
holds in the sense of distributions, with $u^\ep(T,\cdot)= g$,   

\item[(iii)] Equation \eqref{MFG1}-(ii) holds:  
\be\label{eqcontdef}
\ds \quad \partial_t m^\ep-{\rm{div}}(m^\ep D_pH(x,Du^\ep))= 0\  {\rm{in}}\; (0,T)\times \T^d, \quad m^\ep(0)=m_0
\ee
in the sense of distributions,

\item[(iv)] The following equality holds: 
\be\label{defcondsup} \begin{array}{l}
\ds \int_0^T\int_{\T^d} m^\ep(t,x)\left(f(x,m^\ep(t,x))+ H^*(x, D_pH(x,Du^\ep)(t,x)) \right)\dd x\dd t\\
\ds\qquad \qquad \qquad \qquad \qquad \qquad \qquad + \int_{\T^d} m^\ep(T,x)g(x)-m_0(x)u^\ep(0,x)\dd x=0. 
\end{array}\ee 
\end{itemize}
In addition, $u^\ep$ is Hölder continuous and in $W^{1,1}$, and  equality 
$\ds -\partial_t u^\ep+H(x,Du^\ep)=  f(x,m^\ep)$ holds a.e. in $(0,T)\times \T^d$ (see Cardaliaguet-Porretta-Tonon \cite{carporton}). 

Let us now suppose that $f^\ep(x,m)\to f(x,m)$ uniformly with respect to $x$ for any $m\leq \ov m$ and $f^\ep(x,m)\to +\infty$ uniformly with respect to $x$ for any $m> \ov m$ as $\ep\to 0^+$. 

\begin{proposition} Under the above assumptions, 
\begin{enumerate}
\item the family $(u^\ep)$ is bounded in $BV([0,T]\times\T^d)\cap L^r([0,T]\times\T^d)$ while $(D u^\ep)$ is bounded in $L^r([0,T]\times\T^d)$, the family $(\alpha^\ep:=-\partial_tu^\ep+H(\cdot, Du^\ep))$ is bounded in $L^1([0,T]\times\T^d)$, with $\alpha^\ep\geq 0$ a.e.,  $(m^\ep)$ is bounded in $L^\theta([0,T]\times \T^d)$ while $(w^\ep)$ is bounded in $L^{r'}([0,T]\times \T^d)$. 

\item If $(u,m,\a)$ is any cluster point for the weak convergence of $(u^\ep,m^\ep,\a^\ep)$, then $\a\geq f(\cdot, m)$ and, if we set $\beta:= \alpha\mres (0,T)\times \T^d$ and $\beta_T:= u(T^-,\cdot)-g$, then the quadruplet $(u,m, \beta,\beta_T)$ is a solution of the MFG system with density constraint \eqref{MFGdensity}. 
\end{enumerate} 
\end{proposition} 

\begin{proof} The proof is a straightforward adaptation of our previous constructions.  According to \cite{cargraporton}, we know that $(u^\ep,\alpha^\ep)$ is minimizer over ${\mathcal K}_R$ of the functional
$$
{\mathcal A}^\ep(u,\alpha)= \int_0^T \int_{\bb{T}^d} (F^{\ep})^*(x,\alpha)\dd x\dd t - \int_{\bb{T}^d}  u(0^+,x)\dd m_0(x),
$$
where $\ds F^\ep(x,m):= \int_0^m f^\ep(x,s)\dd s$ and $(F^{\ep})^*$ is the Fenchel conjugate of $F^\ep$ with respect to the last variable. 
Then, by convexity,  
$$
(F^{\ep})^*(x,\alpha)\geq \alpha \ov m -F^\ep(x,\ov m),
$$
where, by our assumptions, $F^\ep(x,\ov m)$ converges uniformly with respect to $x$ to $F(x,\ov m)$. Let $\psi$ be  the Lipschitz continuous viscosity  solution to 
$$
\left\{
\begin{array}{rcl}
-\partial_t \psi + H(x,D\psi) & = & 0\\
\psi(T,x) &=&  g(x).
\end{array}
\right.
$$
It is also an a.e. solution, so that $(\psi, 0)$ belongs to ${\mathcal K}_R$. Then 
$$
{\mathcal A}^\ep(u^\ep,\alpha^\ep)\leq {\mathcal A}^\ep(\psi,0)\leq - \int_{\bb{T}^d}  \psi(0,x)\dd m_0(x) \leq C.
$$
So $({\mathcal A}^\ep(u^\ep,\alpha^\ep))$ is bounded from above and one can then argue exactly as in the proof of Lemma \ref{lem:esti} to obtain the bounds on $(u^\ep)$ and $(\alpha^\ep)$ as well as a bound for $({\mathcal A}^\ep(u^\ep,\alpha^\ep))$. 

Following \cite{cargraporton}, we also know that the pair $(m^\ep,w^\ep):= (m^\ep, -m^\ep D_pH(\cdot, Du^\ep))$ is a minimizer over ${\mathcal K}_D$ of 
$$
\s{B}^\ep(m,w) = \int_{\bb{T}^d}  g(x)m(T,x)\dd x + \int_0^T \int_{\bb{T}^d} m(t,x)L\left(x,\frac{w(t,x)}{m(t,x)}\right) + F^\ep(x,m(t,x))\, \dd x \dd t
$$
Since, by \cite{cargraporton}, ${\mathcal A}^\ep(u^\ep,\a^\ep)= - {\mathcal B}^\ep (m^\ep, w^\ep)$, $({\mathcal B}^\ep (m^\ep, w^\ep))$ is bounded. 
From our assumption on $f^\ep$ we have therefore that  $(m^\ep)$ is bounded in $L^\theta([0,T]\times \T^d)$ while $(w^\ep)$ is bounded in $L^{r'}([0,T]\times \T^d)$. 

Let finally $(u,\alpha)$ be a cluster point of $(u^\ep,\a^\ep)$ and $(m,w)$ be a cluster point of $(m^\ep,w^\ep)$ for the weak convergence. Then, standard arguments from the theory of $\Gamma$-convergence show that $(u,\alpha)$ minimizes ${\mathcal A}$ while $(m,w)$ minimizes ${\mathcal B}$, so that, if we set $\beta:= \alpha \mres (0,T)\times \T^d$ and $\beta_T:= u(T^-,\cdot)-g$, the quadruplet $( u,m,\beta,\beta_T)$ is a solution of the MFG system \eqref{MFGdensity} according to Theorem \ref{thm:sol}. 
\end{proof}

\section{No coupling, space homogeneity, power-like Hamiltonians and $m_0<\ov m$}\label{sec:example}

In this section, we study through an example some finer properties of the solutions of \eqref{MFGdensity}.  Let us consider  $f(x,m)= 0,$ for all $(x,m)\in\T^d\times[0,+\infty),$ $\ds H(x,p)=\frac{1}{s}|p|^s$ ($s>1$) and $T=1.$ The terminal cost $g$ is a given smooth function. As usual, we assume that the initial density of the population satisfies $m_0<\ov m-c$ a.e. in $\T^d$ for a given constant $0<c<\ov m$ (here this assumption will be essential, while it is not clear whether for the considerations of the previous sections it is purely technical or not). For simplicity, let us set $s=2$. In this  case the functional $\cB$ for the Problem \ref{pr:dual} has the form
$$\cB(m,w)=\int_0^1\int_{\T^d}\frac12\frac{|w|^2}{m} + F(x,m)\dd x\dd t + \int_{\T^d}g(x)m(1,x)\dd x,$$
where we use the convention \eqref{conventionH*}. Let us also chose $F(x,m)\equiv 0$ for $m\in[0,\ov m]$ and $F\equiv+\infty$ otherwise. This functional recalls the one introduced by Bemamou and Brenier to give a dynamical formulation for the Monge-Kantorovich's optimal transportation problem (see \cite{bb}). Only a constraint on the density $m$ and a penalization on the final value have been added.

Indeed, forgetting for a while the density constraint, Problem \ref{pr:dual} can be reformulated as 
\begin{equation}\label{pr:bb}
\min\left\{\frac{1}{2}W_2^2(m_0,m_1)+\int_{\T^d}gm_1\dd x\; : \; m_1\in\cP(\T^d),\; m_1\le\ov m\right\}.
\end{equation}
We remark that the above formulation gives always a geodesic curve connecting $m_0$ and $m_1$ (thus $m_t$ is defined for all $t\in[0,1]$). Since the admissible set in the above problem is geodesically convex (and $\T^d$ is a convex set), the density constraint is satisfied as soon as it is satisfied at the terminal time. Hence the problem in \eqref{pr:bb} is completely equivalent to Problem \ref{pr:dual}. Actually we can prove something more: if the initial density satisfies strictly the constraint, then saturation may happen only at the final time. This result is not a straightforward consequence of geodesic convexity, and we give a complete proof of it here below. 

\begin{lemma}\label{lem:sat}
Let $m_0<\ov m-c$ (for a given constant $0<c<\ov m$) a.e. in $\T^d$ and $m_1$ be the solution of Problem \ref{pr:bb}. Let $(m_t)$ be the  geodesic connecting $m_0$ to $m_1$. Then, for any $t\in (0,1)$ we have $\|m_t\|_{L^\infty}\leq \ov m \frac{\l}{(1-t)+t\lambda^{1/d})^d}$, where $\l:=(\ov m-c)/\ov m<1$ (note $\frac{\l}{(1-t)+t\lambda^{1/d})^d}<1$ for $t<1$).\end{lemma}
\begin{proof}
As we mentioned before, since the admissible set in Problem \ref{pr:bb} is geodesically convex, we get $m_t\le\ov m$ a.e. in $\T^d$ for all $t\in[0,1].$ On the other hand, since $m_t$ is absolutely continuous for all $t\in[0,1]$ we know that there exist optimal transport maps $T^t,S^t:\T^d\to\T^d$ such that $(T^t)_\# m_0=m_t$ and $(S^t)_\# m_t=m_0$ with $T^t\circ S^t=\id.$ The maps $(S_t)$ and $(T_t)$ are given by McCann's interpolation in terms of $S^1$ and $T^1$ respectively, which is $T_t:=(1-t)\id+tT^1$ and $S_t=t\id+(1-t)S^1.$  Moreover $T_t$ and $S_t$ are countably Lipschitz (i.e. the domain can be decomposed, up to negligible sets, into a countable union of sets where these maps are Lipschitz continuous), hence we can write the Jacobian equation
$$\det(DT_t)=\frac{m_0}{m_t\circ T_t}.$$
Hence, the density $m_t$ is given by
\begin{equation}\label{density}
m_t=\frac{m_0}{\det(DT_t)}\circ S_t.
\end{equation}
Using the concavity of $\det^{1/d}$ (for positive definite matrices) we obtain that 
\begin{align*}
\det(DT_t)=\det\left((1-t)I_d+tDT^1\right)&\ge\left((1-t)+t\det(DT^1)^{1/d}\right)^d=\left((1-t)+t\left(\frac{m_0}{m_1\circ T^1}\right)^{1/d}\right)^d\\
&\ge \left((1-t)+t\left(\frac{m_0}{\ov m}\right)^{1/d}\right)^d.
\end{align*}
Hence by \eqref{density} we have that 
$$m_t\le\frac{m_0\circ S_t}{\left((1-t)+t\left(\frac{m_0\circ S_t}{\ov m}\right)^{1/d}\right)^d}.$$
Let us set $\l:=(\ov m-c)/\ov m<1.$ Then, for any $t\in (0,1)$ we have
$$m_t\le \ov m \frac{\l}{(1-t)+t\lambda^{1/d})^d},$$
and the coefficient $\frac{\l}{(1-t)+t\lambda^{1/d})^d}$ is strictly less than $1$ for every $t,\lambda<1$.
\end{proof}

\subsection{Some properties of $\b,\b_1$ and $u$}
Let us discuss now some further properties of $\b,\b_1$ and $u$.

\begin{proposition}
Let $(u,m,\beta,\beta_1)$ be a solution of the MFG system with density constraints and let us assume that we are in the setting of this section. Then $\beta\equiv 0$ and $u$ and $\beta_1$ are bounded. 
\end{proposition}

\begin{proof} From Theorem \ref{thm:sol}, we know that the pair $(m,-mD_pH(\cdot, Du))$ is a minimizer of ${\mathcal B}$. In view of  Lemma \ref{lem:sat}, we have therefore $m(t,x)<\ov m$ for a.e. $(t,x)\in (0,T)\times \T^d$. By Definition \ref{def:solution}, this implies that $\beta^{ac}=0$. Recall on the other hand that 
$$
\limsup_{\ep\to 0} m_\ep(t,x)= \ov m\qquad \beta^s-{\rm{a.e.}}\; {\rm{if}} \; t<1, 
$$
where $m_\ep$ is any standard mollification of $m$. But, still by Lemma \ref{lem:sat}, for any $t\in (0,1)$, we get an upper bound on  $m_\ep(t,x)$ which is strictly less than $\ov m$. Hence, $\beta^s=0$ on $[0,1)\times\T^d$.

Let us now check that $u$ is bounded. 
We note that, as $\beta=0$ and $H$ satisfies the growth condition \eqref{eq:hamiltonian_bounds}, $u$ satisfies a.e. the inequality $-\partial_t u +|Du|^2/2\leq 0$ in $(0,1)\times \T^d$. Thus, if we mollify $u$ in the usual way, $u_\ep$ is a classical sub solution to $-\partial_t u_\ep +|Du_\ep|^2/2\leq  0$ on $(\ep,1-\ep)\times \T^d$. By Hopf's formula we get therefore
$$
u_\ep(t,x)\leq \inf_{y\in \T^d} \left\{ u_\ep(1-\ep,y)+ C \frac{|x-y|^{2}}{(1-\ep-t)} +C(1-t)\right\}\qquad \forall (t,x)\in (\ep, 1-\ep)\times \T^d.
$$
Hence 
$$
u_\ep(t,x)\leq \inf_{y\in \T^d} \left\{ u_\ep(1-\ep,y)+ C\right\}\qquad \forall (t,x)\in (\ep, 1/2)\times \T^d.
$$
Recalling that $\ds \int_{\T^d} u(t,x)\dd x$ is bounded for a.e. $t$ (see the proof of Lemma \ref{lem:esti}), we also have that $\ds\int_{\T^d} u_\ep(t,x)\dd x$ is bounded as well for all $t$ and therefore 
$\ds \inf_{y\in \T^d} u_\ep(1,y)$ is bounded from above. So we have proved that $u_\ep$ is bounded from above by a constant $C_0$ on $(\ep,1/2)\times \T^d$, where $C_0$ is independent of $\ep$. This shows that $u$ is bounded from above by $C_0$ on  $(0,1/2)\times \T^d$.

Let us set $z(t,x):= (C_0+\|H(\cdot,0)\|_{L^\infty})\vee \|g\|_{L^\infty}- \|H(\cdot,0)\|_{L^\infty} (1-t)$. Then $z$ is a subsolution to $-\partial_t z+H(x,Dz)\leq 0$ which satisfies $z(1, \cdot)\geq g$ and $z(0,\cdot)\geq C_0\geq u(0,\cdot)$. Therefore the map $\tilde u(t,x):= u(t,x)\wedge z(t,x)$ is still a subsolution (because $H=H(x,p)$ is convex with respect to $p$), which satisfies $\tilde u(0,\cdot)=u(0,\cdot)$ a.e. and $g(x)\leq \tilde u(1^-,x)\leq u(1^-,x)$. Let us set $\tilde \alpha:= (\tilde u(1^-, \cdot)-g)\dd(\d_1\otimes{\mathcal H}^d\mres\T^d)$. Then the pair $(\tilde u,\tilde \alpha)$ belongs to ${\mathcal K}_R$ and by optimality of $(u,\alpha)$ we have  
$$
\begin{array}{l}
\ds {\mathcal A}(u,\alpha) = \int_{\T^d} \left(u(1^-,x)-g(x)\right)\dd x - \int_{\T^d} m_0(x) u(0,x)\dd x
\\ 
\ds \qquad \leq 
{\mathcal A}(\tilde u,\tilde \alpha) = \int_{\T^d}\left( \tilde u(1^-,x)-g(x)\right) \dd x - \int_{\T^d} m_0(x) \tilde u(0,x)\dd x.
\end{array}
$$
As $\tilde u(0,\cdot)=u(0,\cdot)$ and $\tilde u(1^-,x)\leq u(1^-,x)$, this proves that $\tilde u(1^-,x)= u(1^-,x)$ a.e., which means that $u(1^-,\cdot)$ is bounded from above. Since we already know that $u$ is bounded from below (see the proof of Lemma \ref{lem:esti}), we have established that $u(1^-,\cdot)$ is bounded. 
By Hopf's formula, this entails the boundedness of $u$ on $(0,1)\times \T^d$, from where the boundedness of $\beta_1$ follows as well.
\end{proof}

\begin{remark}[Nash-type equilibrium]
For this example a notion of Nash equilibrium can be formulated by the means of $(m,\b_1),$ i.e. by the means of the ``additional price'' $\b_1$ to be payed by the agents at the final time. This price, whose value is precisely $\b_1=(u(1^-,\cdot)-g)$, clearly has to be payed only if agents arrive to the saturated zone at the final time. Let us postpone the precise definition and the  details on the question of the Nash equilibrium, which will be established for more general cases in Section \ref{sec:nash} (see Definition \ref{def:nash}).
\end{remark}

\section{Regularity of the ``additional price'' $\b$}\label{sec:regularity}

In this section we show, under some additional regularity assumption on the data, that the measure $\beta$ is absolutely continuous and belongs to $L^{d/(d-1)}_{\rm{loc}}((0,T)\times\T^d)$.  In this respect, our model recalls those studied by Brenier (see \cite{br}) and later by Ambrosio-Figalli (see \cite{af, af2}), where they analyzed the motion of incompressible perfect fluids driven by the Euler's equations.

We will see in the next section that this regularity is essential in order to define Nash equilibria in our context. For this, we assume in addition to the previous hypotheses the following conditions: there exists $\lambda>0$ such that
\begin{enumerate} 
\item[\bf{(HP1)}] (Assumption for $H$): $H$ and $H^*$ are of class $C^2$ with
\be\label{hypH31}
\lambda I_d \leq D^2_{pp}H\leq \lambda^{-1}I_d\;\;{\rm{and}}\;\;  \lambda I_d \leq D^2_{qq}H^*\leq \lambda^{-1}I_d,
\ee
\be\label{hypH3}
|D^2_{xx}H^*(x,p)|\leq C(1+|p|^2), \qquad |D^2_{xp}H^*(x,p)|\leq C(1+|p|).
\ee

\item[\bf{(HP2)}] (Assumption on $F$): $f$ is of class $C^2$ on $\T^d\times [0,\ov m]$ and, for any $m\in [0,\ov m]$ and $\alpha\geq 0$, 
\be\label{hyp:ineqF+F*}
F(x,m)+F^*(x,\alpha) -\alpha m \geq \frac{\lambda}{2} | \alpha_1-f(x,m)|^2 + p(\ov m-m),
\ee
where 
$p=(\alpha-f(x,\ov m))_+$ and $\alpha_1= \alpha-p$. 
\item[\bf{(HP3)}] (Assumption on $g$): $g$ is of class $C^2$.
\end{enumerate}

\begin{remark}
Observe that the assumption {\bf(HP2)} is fulfilled if $\partial^2_{\alpha\alpha}F^*(x,\alpha)\geq \lambda$ on $(0, f(x,\ov m))$ for some $\lambda>0$. This assumption of course holds if $f(x,\ov m)=0$. If $f(x,\ov m)>0$, since $\partial_\alpha F^*(x,f(x,m))=m$ on $(0,\ov m)$, the Implicit Function Theorem implies that $\partial^2_{\alpha\alpha}F^*(x,\alpha)= 1/\partial_mf(x,m)$, which means that the assumption is indeed satisfied as soon as $\partial_mf(x,m)$ is bounded from above on $(0,\overline m)$, which makes it a very natural assumption. Among the examples that we have in mind, we underline the case where $H(x,p)=|p|^2/2-f(x)$ and $F(x,m)=0$ if $m\in[0,\ov m]$ and $+\infty$ otherwise. Notice that the same example could be written with $F(x,m)=f(x)m$ if $m\in[0,\ov m]$ and $+\infty$ otherwise and $H(x,p)=|p|^2/2$ (these different expression give rise to the same global and individual problems), but in this case {\bf(HP2)} would not be satisfied (and also {\bf(H4)} would be violated).
\end{remark}

\begin{theorem}\label{regularity_pressure} Let $(u,m,\beta,\beta_T)$ be a solution of the MFG system \eqref{MFGdensity}. 
Under the above assumptions, $f(\cdot,m(\cdot,\cdot)) \in H^1_{\rm{loc}}((0,T)\times \T^d)$ and $\beta$ is absolutely continuous in $(0,T)\times \T^d$ with $$\beta\in L^2_{\rm{loc}}((0,T);BV(\T^d))\hookrightarrow L^{d/(d-1)}_{\rm{loc}}((0,T)\times \T^d).$$ 
\end{theorem}

As we said, the proof is largely inspired by the works of Brenier (see \cite{br}) and Ambrosio-Figalli (see \cite{af}) on the incompressible Euler's equations.

\begin{proof}[Proof of the Theorem \ref{regularity_pressure}] 
By abuse of notion, we use $\cB(m',v')$ meaning $\cB(m',m'v')$ for any admissible pair $(m',m'v')$ in the dual problem ($v'$ denoting the velocity field). 

Throughout the proof, $(u,m,\beta,\beta_T)$ is a fixed solution of the MFG system \eqref{MFGdensity} and we define $\alpha$ by \eqref{def.alpha} and set
$w= -mD_pH(x, Du)$. Recall that $(m,w)$ is a minimizer  for $\cB$. We also set  $v:=w/m$ and construct competitors $(m^{\d,\eta},m^{\d,\eta}v^{\d,\eta})$ in the following  way: let us fix  $0<t_1<t_2<T$ and let $\zeta\in C_c^{\infty}((0,T);[0,1])$ be a smooth cut-off such that $\zeta\equiv 1$ on $[t_1,t_2];$
for $\eta>0$ small and $\d\in\R^d$ small (such that $t+\zeta(t)\eta\in(0,T)$ for all $t\in[0,T]$), we denote 
$$m^{\d,\eta}(t,x):=m(t+\zeta(t)\eta,x+\zeta(t)\d)$$ 
the time-space translation of the density and let 
$$v^{\d,\eta}(t,x):=v(t+\zeta(t)\eta,x+\zeta(t)\d)(1+\eta\zeta'(t))-\zeta'(t)\d$$
the velocity field associated to $m^{\d,\eta}.$ Indeed, by construction $(m^{\d,\eta},m^{\d,\eta}v^{\d,\eta})$ solves the continuity equation, and satisfies the other constraints.

{\it Step 0.}  Let us collect some tools now.

{\it First}, we have
\be\label{eta_delta}
 {\mathcal B}(m^{\delta,\eta},v^{\delta,\eta})
\leq {\mathcal B}(m,v)+ C(\eta^2+|\delta|^2).
\ee
Indeed, let us denote by $\xi_\eta$ the inverse map of $t\mapsto t+\zeta(t)\eta$. Then, after changing variables,  
$$
\begin{array}{l}
\ds  {\mathcal B}(m^{\delta,\eta},v^{\delta,\eta}) \\ 
\qquad =   \ds 
 \int_0^T \int_{\T^d} \left[m(s,y) H^*\left(y-\zeta(\xi_\eta(s))\delta, -v(s,y)(1+\eta\zeta'(\xi_\eta(s)))+\zeta'(\xi_\eta(s))\delta)\right) \right. \\
\qquad \qquad \ds  \qquad \qquad + \left. F(y-\zeta(\xi_\eta(s))\delta, m(s,y))\right]  \xi_\eta'(s) \dd y \dd s
 +\int_{\T^d} g(y) m(T, y) \dd y.
 \end{array}
 $$
 In view of our $C^2$ regularity assumptions on $H^*$, $F$ and $g$, the map $(\eta,\delta)\mapsto {\mathcal B}(m^{\delta,\eta},v^{\delta,\eta})$ is $C^2$. We obtain \eqref{eta_delta} by optimality of $(m,v)$.

{\it Second}, by stationarity of the problem for ${\mathcal B}$ (it is enough to consider perturbations of form $(m^{0,\eta},v^{0,\eta})$ for $\zeta$ with compact support, not necessarily 1 on $[t_1,t_2]$), we have 
$$
\int_{\T^d} \left\{m(H^*(x,-v(t,x))+ D_qH^*(x,-v(t,x))\cdot v(t,x))+F(x,m(t,x))\right\}\dd x= constant.
$$
From our assumption on $H^*$, we have 
$$
H^*(x,-v)+D_qH^*(x,-v)\cdot v\leq H^*(x,0)-\frac{\lambda}{2}|v|^2.
$$
Thus  
\begin{equation}\label{kin_en1}
\ds{\rm{ess-sup}}_{t\in [0,T]} \int_{\T^d} m(t,x)|v(t,x)|^2\dd x\leq C.
\end{equation}
By \eqref{hypH31}, we have $D^2_{qq}H^*\leq (1/\lambda) I_d$ and therefore \eqref{kin_en1} implies
\begin{equation}\label{kin_en2}
\ds{\rm{ess-sup}}_{t\in [0,T]} \int_{\T^d}m(t,x)|D_qH^*(x,-v(t,x))|^2\dd x\le C.
\end{equation}
 
{\it Third}, for any smooth map $(u',\alpha')$, with $\alpha'\geq 0$, and $(m',w')\in\cK_D$ (where $v'=w'/m'$) competitor for the primal and the dual problems respectively, we have  
\begin{align*}
{\mathcal A}(u',\alpha') +{\mathcal B}(m',v') \geq& 
\int_0^T\int_{\T^d} \left\{m'(H(x,Du')+H^*(x,-v')+ v'\cdot Du')\right\}\dd x\dd t\\ 
&+ \int_0^T\int_{\T^d} \left\{F(x,m')+F^*(x,\alpha')-\alpha' m'\right\}\dd x\dd t.
\end{align*}
In view of our assumptions on {\bf (HP1)} and {\bf (HP2)}, we have 
the key inequality 
\be\label{estifond}
\begin{array}{rl}
\ds {\mathcal A}(u',\alpha') +{\mathcal B}(m',v')  \geq & \ds 
\frac{\lambda}{4} \int_0^T\int_{\T^d} m'(t,x)|Du'-D_qH^*(x,-v')|^2\dd x\dd t\\[10pt]
& +\ds \frac{\lambda}{4} \int_0^T\int_{\T^d} m'(t,x)|v'+D_pH(x,Du')|^2\dd x\dd t\\[10pt]
& \ds + \int_0^T\int_{\T^d}\left\{ \frac{\lambda}{2}| \alpha_1'-f(x,m')|^2 + p'(\ov m-m')\right\}\dd x\dd t
\end{array}
\ee
where 
$p'=(\alpha'-f(x,\ov m))_+$ and $\alpha_1'= \alpha'-p'$. 

With the help of these tools let us show now the statements of the theorem. 

\bigskip
{\it Step 1.} We first check that $f(\cdot,m) \in H^1_{\rm{loc}}((0,T)\times \T^d)$.  Let us fix $(m',v')$ to be a smooth competitor for $\cB$ and let $(u_n,\alpha_n)$ be the minimizing sequence for Problem \ref{pr:relaxed} defined in Lemma \ref{lem.approx}: we know that $(\alpha_n)$ is bounded in $L^1$ and converges to the (nonnegative) measure $\alpha$ defined from $(\beta,\beta_T)$ by \eqref{def.alpha}. Then, passing to the limit in the inequality
$$
{\mathcal A}(u_n,\alpha_n)+ {\mathcal B}(m',v') \geq
\int_0^T\int_{\T^d} \left\{F(x,m')+F^*(x,\alpha_n)-\alpha_n m'\right\}\dd x\dd t,
$$
we get
$$
\inf_{\cK_R} {\mathcal A}+ {\mathcal B}(m',v')  \geq
\int_0^T\int_{\T^d} \left\{F(x,m')+F^*(x,\alpha^{ac})-\alpha^{ac} m'\right\}\dd x\dd t+ \int_0^T\int_{\T^d}  (\ov m-m')\dd\alpha^s.
$$
In view of the proof of Theorem \ref{thm:sol}, we have $\alpha^{ac}\geq f(\cdot,m)$, with an equality in $\{m<\ov m\}$. 
So, if we set as above $p=(\alpha^{ac}-f(x,\ov m))_+$ and $\alpha^{ac}_1= \alpha^{ac}-p$, then $\alpha^{ac}_1=f(\cdot, m)$. By \eqref{hyp:ineqF+F*}, this implies that 
$$
\inf_{\cK_R} {\mathcal A}+ {\mathcal B}(m',v')  \geq
\int_0^T\int_{\T^d} \frac{\lambda}{2} | f(x,m)-f(x,m')|^2\dd x\dd t,
$$
an inequality which remains true for any $(m',v')\in\cK_D$ (not necessarily smooth ones). 
Adding $\ds\inf_{\cK_R}\cA$ to inequality \eqref{eta_delta} and using the duality $\ds\inf_{\cK_R}\cA+\min_{\cK_D}\cB=0$ we have  
\be\label{estim_imp}
\ds\inf_{\cK_R} {\mathcal A}+ {\mathcal B}(m^{\delta,\eta},v^{\delta,\eta})\leq C(\eta^2+|\delta|^2), 
\ee
which implies
$$
\int_{t_1}^{t_2}\int_{\T^d}  | f(x,m)-f(x,m^{\delta,\eta})|^2\dd x\dd t\leq C(\eta^2+|\delta|^2)
$$
and the regularity of $f$ in $x$ allows to conclude $f(\cdot,m)\in H^1_{\rm{loc}}((0,T)\times\T^d).$
\bigskip

{\it Step 2.} Let $(u_n,\alpha_n)$ be the minimizing sequence defined by Lemma \ref{lem.approx}. Without loss of generality, we can assume that
\begin{equation}\label{min_seq}
{\mathcal A}(u_n,\alpha_n)-\inf_{\cK_R}\cA\leq 1/n. 
\end{equation}
We set $p_n:= (\alpha_n-f(\cdot,\ov m))_+$ and $\alpha_{1,n}:= \alpha_n-p_n$. For $\varphi:[0,T]\times\T^d\to\R$ and for $\eta>0$ small, let us define the average of $\varphi$ on the $[t,t+\eta]$ by 
$$\varphi^\eta(t,x):=\int_0^1\varphi(t+\theta\eta,x)\dd\theta,$$
which is well-defined on $[t_1,t_2]\times\T^d.$ With this procedure, we similarly define the functions $p_{n}^\eta$, $\alpha_{n}^{\eta}$, etc. Let us take moreover $\sig\in C^\infty([t_1,t_2];[0,+\infty)).$

The aim of this step consists in estimating the quantity
$$
I:=\int_{t_1}^{t_2} \int_{\T^d} \sig(t) \ov m\, |p_{n}^{\eta}(t,x+\delta)-p_{n}^{\eta}(t,x)|\dd x\dd t.
$$
Namely, we prove that
\be\label{ineq.I}
\begin{array}{rl}
\ds I&\ds \leq C\|\sig\|_{L^2}\left(|\delta|+ \left(1+ \frac{|\delta|}{\eta}\right) (1/n+ |\delta|^2+\eta^2)^{1/2}\right)\\
&\ds+C\|\sig\|_{L^\infty}(1/n+\eta^2+ |\delta|^2)^{1/2}\left[(1/n+\eta^2+ |\delta|^2)^{1/2}+|\d|\left(1+(1/n+\eta^2+ |\delta|^2)^{1/2}\right)\right]\\
&\ds+C\left\{\|\sig\|_{L^2}^2+\|\sig\|_{L^\infty}\left[(1/n+|\d|^2+\eta^2)+(1/n+|\d|^2+\eta^2)^{1/2}\right]\right\}^{1/2}(1/n+|\d|^2+\eta^2)^{1/2}\\
&\ds =: X(\sig,1/n,\d,\eta).
\end{array}
\ee
We will show in the last two steps that this inequality easily entails the desired estimates on $p$ and $\beta$. 

The proof of \eqref{ineq.I} is quite long and relies on the combination of \eqref{estifond}, \eqref{estim_imp} and \eqref{min_seq} which imply that 
\be\label{estifondBIS}
\begin{array}{rl}
\ds 1/n+C(\eta^2+|\delta|^2)  \geq & \ds 
\frac{\lambda}{4} \int_{t_1}^{t_2}\int_{\T^d} m^{\delta,\eta}(t,x)|Du_n-D_qH^*(x,-v^{\delta,\eta})|^2\dd x\dd t\\[10pt]
& +\ds \frac{\lambda}{4} \int_{t_1}^{t_2}\int_{\T^d} m^{\delta,\eta}(t,x)|v^{\delta,\eta}+D_pH(x,Du_n)|^2\dd x\dd t\\[10pt]
& \ds + \int_{t_1}^{t_2}\int_{\T^d}\left\{ \frac{\lambda}{2}| \alpha_{1,n}-f(x,m^{\delta,\eta})|^2 + p_n(\ov m-m^{\delta,\eta})\right\}\dd x\dd t
\end{array}
\ee

We have 
\begin{align*}
I&\leq \int_{t_1}^{t_2} \int_{\T^d} \sig(t) (\ov m -m(t,x)) |p_{n}^{\eta}(t,x+\delta)-p_{n}^{\eta}(t,x)|\dd x\dd t\\
&+\int_{t_1}^{t_2} \int_{\T^d} \sig(t) m(t,x) |p_{n}^{\eta}(t,x+\delta)-p_{n}^{\eta}(t,x)|\dd x\dd t\\
&=:I_{01}+I_{02}
\end{align*}
where the the first term can be estimated as follows:
\begin{align*}
I_{01}&\le  \|\sig\|_{L^{\infty}}\int_{t_1}^{t_2} \int_{\T^d}  (\ov m -m(t,x))\left\{ |p_{n}^{\eta}(t,x+\delta)|+|p_{n}^{\eta}(t,x)|\right\}\dd x\dd t\\
&=\|\sig\|_{L^{\infty}}\int_0^1\dd\theta\int_{t_1}^{t_2} \int_{\T^d}(\ov m -m(t,x))\left\{p_n(t+\theta\eta,x+\d)+p_n(t+\theta\eta,x)\right\}\dd x\dd t\\ 
&\le\|\sig\|_{L^{\infty}}\int_0^1\dd\theta\int_{0}^{T} \int_{\T^d}(\ov m -m^{-\d,-\theta\eta})p_n\dd x\dd t\\
&+\|\sig\|_{L^{\infty}}\int_0^1\dd\theta\int_{0}^{T} \int_{\T^d}(\ov m -m^{0,-\theta\eta})p_n\dd x\dd t.
\end{align*}
Now by  \eqref{estifondBIS} we
obtain that 
$$I_{01}\le C\|\sig\|_{L^\infty}(1/n+|\d|^2+\eta^2).$$

For the second term we have 
\begin{align*}
I_{02}\le&\int_{t_1}^{t_2} \int_{\T^d} \sig(t) m(t,x) |\alpha_{n}^{\eta}(t,x+\delta)-\alpha_{n}^{\eta}(t,x)|\dd x\dd t\\
&+\int_{t_1}^{t_2} \int_{\T^d} \sig(t) \ov m |\alpha_{1,n}^{\eta}(t,x+\delta)-\alpha_{1,n}^{\eta}(t,x)|\dd x\dd t\\
&:=I_1+I_2.
\end{align*}
To estimate the term $I_1$, let us compute 
\begin{align*}
\ds \alpha_{n}^{\eta}(t,x+\delta)-\alpha_{n}^{\eta}(t,x)&= \ds \int_0^1 -\partial_t u_n(t+\theta\eta,x+\delta)+H(x+\delta, Du_n(t+\theta \eta, x+\delta))\dd\theta \\
&- \int_0^1 -\partial_t u_n(t+\theta\eta,x)+H(x, Du_n(t+\theta \eta, x))\dd\theta \\
&= \ds - \eta^{-1}\int_0^1 [ Du_n(t+\eta,x+s\delta)- Du_n(t,x+s\delta)]\cdot\delta \dd s
\\
& +\int_0^1\int_0^1  D_xH(x+s\delta, Du_n(t+\theta \eta, x+s\delta))\cdot \delta  \ ds d\theta\\
&+\int_0^1\dd\theta\int_0^1   D_pH(x+s\delta, \xi_s)\cdot [Du_n(t+\theta \eta, x+\delta)-Du_n(t+\theta \eta, x) ] \dd s \dd\theta
\end{align*}
where $\xi_s:= (1-s) Du_n(t+\theta \eta, x)+s Du_n(t+\theta \eta, x+\delta)$. 
Thus,
\begin{align*}
 |\alpha_{n}^{\eta}(t,x+\delta)-\alpha_{n}^{\eta}(t,x)|&\le |\delta| \eta^{-1}\int_0^1 |Du_n(t+\eta,x+s\delta)- Du_n(t,x+s\delta)|\dd s \\
& + |\delta| \int_0^1\int_0^1  | D_xH(x+s\delta, Du_n(t+\theta \eta, x+s\delta))| \dd s\dd\theta\\
& +\int_0^1\int_0^1 |D_pH(x+s\delta, \xi_s)|\  |Du_n(t+\theta \eta, x+\delta)-Du_n(t+\theta \eta, x)| \dd s\dd\theta.
\end{align*}

In view of our assumption \eqref{hypH3} on $D_xH$ and $D_pH$: 
\begin{align*}
I_1&= \int_{t_1}^{t_2}\int_{\T^d} \sig(t) m(t,x)|\alpha_{n}^{\eta}(t,x+\delta)-\alpha_{n}^{\eta}(t,x)|\dd x\dd t \\ 
 &\leq |\delta| \eta^{-1} \int_{t_1}^{t_2}\int_{\T^d} \int_0^1 \sig(t) m(t,x)  |Du_n(t+\eta,x+s\delta)- Du_n(t,x+s\delta)|\dd s\dd x\dd t \\
&+ C|\delta| \int_{t_1}^{t_2}\int_{\T^d} \int_0^1 \int_0^1 \sig(t) m(t,x)  \left\{1+|Du_n(t+\theta \eta, x+s\delta)|^2\right\}\dd s\dd\theta\dd x\dd t\\
&+C\int_{t_1}^{t_2}\int_{\T^d} \int_0^1 \sig(t) m(t,x) \left\{1+ |Du_n(t+\theta \eta, x)|+ |Du_n(t+\theta \eta, x+\delta)|\right\}\\
&\ds \qquad \qquad \qquad\qquad \qquad \qquad\qquad\qquad  \times   |Du_n(t+\theta \eta, x+\delta)-Du_n(t+\theta \eta, x)| \dd\theta\dd x\dd t \\
&:= I_{11}+I_{12}+I_{13}.
\end{align*}

For $I_{11}$, we have 
\begin{align*}
 I_{11}& \leq |\delta| \eta^{-1}\int_{t_1}^{t_2}\int_{\T^d} \int_0^1\bigg\{ \sig(t) m(t,x) \big( |Du_n(t+\eta,x+s\delta)-D_qH^*(x,-v(t, x))| \\
&  \hspace{3.5cm}+ | D_qH^*(x,-v(t, x))- Du_n(t,x+s\delta)|\big)\bigg\}\dd s\dd x\dd t \\
& \leq  |\delta| \eta^{-1}\int_{t_1+\eta}^{t_2+\eta}\int_{\T^d} \int_0^1 \sig^{0,-\eta}(t) m^{-s\delta,-\eta}(t,x) |Du_n(t,x)-D_qH^*(x-s\d,-v^{-s\delta,-\eta}(t, x))| \dd s\dd x\dd t\\
&  +|\delta| \eta^{-1}\int_{t_1}^{t_2}\int_{\T^d} \int_0^1 \sig(t) m^{-s\delta,0}(t,x)|Du_n(t,x)-D_qH^*(x-s\d,-v^{-s\delta,0}(t, x))| \dd s\dd x\dd t
\end{align*}

By Cauchy-Schwarz and \eqref{estifondBIS} we obtain:
$$
I_{11} \leq C|\delta| \eta^{-1}\|\sig\|_{L^2} (1/n+\eta^2+|\delta|^2)^{1/2}.
$$
We now estimate $I_{12}$, that we bound from above as follows: 
\begin{align*}
I_{12}&\le C\|\sig\|_{L^2}|\d| +C|\d| \int_{t_1}^{t_2}\int_{\T^d} \int_0^1 \int_0^1 \sig(t) m(t,x)|D_qH^*(x,-v(t,x))|^2\dd s\dd\theta\dd x\dd t\\
&+C|\delta| \int_{t_1}^{t_2}\int_{\T^d} \int_0^1 \int_0^1 \sig(t) m(t,x)  \left\{|Du_n(t+\theta \eta, x+s\delta)|^2-|D_qH^*(x,-v(t,x))|^2\right\}\dd s\dd\theta\dd x\dd t
\end{align*} 
The second term can be estimated by \eqref{kin_en2}, while, for the third one, we use the inequality $|a|^2-|b|^2\leq |a-b|^2+|a-b||b|$ to get:
\begin{align*}
I_{12}&\le C\|\sig\|_{L^2}|\d|+ \\
&+C\|\sig\|_{L^{\infty}}|\delta| \int_{t_1}^{t_2}\int_{\T^d} \int_0^1 \int_0^1 m(t,x)|Du_n(t+\theta \eta, x+s\delta)-D_qH^*(x,-v(t,x))|^2\dd s\dd\theta\dd x\dd t\\ 
& +2C\|\sig\|_{L^{\infty}}|\delta| \int_{t_1}^{t_2}\int_{\T^d} \int_0^1 \int_0^1\bigg\{m(t,x)|Du_n(t+\theta \eta, x+s\delta)-D_qH^*(x,-v(t,x))|\\ 
&\hspace{4.5cm}\times|D_qH^*(x,-v(t,x))|\bigg\}\dd s\dd\theta\dd x\dd t\\
\end{align*}
\begin{align*}
\ \ &\le C\|\sig\|_{L^2}|\d|+ \\
&+C\|\sig\|_{L^{\infty}}|\delta| \int_{t_1+\theta\eta}^{t_2+\theta\eta}\int_{\T^d} \int_0^1 \int_0^1 m^{-s\d,-\theta\eta}(t,x)|Du_n(t, x)-D_qH^*(x-s\d,-v^{-s\d,-\theta\eta}(t,x))|^2\dd s\dd\theta\dd x\dd t\\
& +2C\|\sig\|_{L^{\infty}}|\delta| \int_{t_1+\theta\eta}^{t_2+\theta\eta}\int_{\T^d} \int_0^1 \int_0^1\bigg\{m^{-s\d,-\theta\eta}(t,x)|Du_n(t, x)-D_qH^*(x-s\d,-v^{-s\d,-\theta\eta}(t,x))|\\ 
&\hspace{4.5cm}\times |D_qH^*(x-s\d,-v^{-s\d,-\theta\eta}(t,x))|\bigg\}\dd s\dd\theta\dd x\dd t.
\end{align*}
As before, using the energy estimates \eqref{kin_en2} and \eqref{estifondBIS}  together with a Cauchy-Schwarz inequality in the last integral we obtain
$$I_{12}\le C\|\sig\|_{L^2}|\d|+C\|\sig\|_{L^{\infty}}|\d|\left\{(1/n+\eta^2+|\d|^2)+C(1/n+\eta^2+|\d|^2)^{1/2}\right\}.$$
It is easy to see that with the help of the estimations for $I_{11}$ and $I_{12}$ we can estimate $I_{13}$ as well. Hence we obtain
\begin{align*}
I_{13}&\le C\|\sig\|_{L^2}(1/n+|\d|^2+\eta^2)^{1/2}\\
&+C\left\{C\|\sig\|_{L^2}^2+C\|\sig\|_{L^\infty}\left[(1/n+|\d|^2+\eta^2)+(1/n+|\d|^2+\eta^2)^{1/2}\right]\right\}^{1/2}(1/n+|\d|^2+\eta^2)^{1/2}.
\end{align*}

\bigskip

Let us now take care of $I_2$. Setting $\ds f^\eta(t,x):= \int_0^1f(x,m(t+\theta\eta,x))\dd\theta$, we have
\begin{align*}
I_2&= \int_{t_1}^{t_2} \int_{\T^d} \sig(t) \ov m |\alpha_{1,n}^{\eta}(t,x+\delta)-\alpha_{1,n}^{\eta}(t,x)|\dd x\dd t\\
& \leq   \int_{t_1}^{t_2} \int_{\T^d} \sig(t) \ov m |\alpha_{1,n}^{\eta}(t,x+\delta)-f^\eta(t,x+\delta)|\dd x\dd t +\int_{t_1}^{t_2} \int_{\T^d} \sig(t) \ov m |f^\eta(t,x+\delta)-f^\eta(t,x)|\dd x\dd t \\
&  + \int_{t_1}^{t_2} \int_{\T^d} \sig(t) \ov m |f^\eta(t,x)-\alpha_{1,n}^{\eta}(t,x)|\dd x\dd t\\ 
& :=\; I_{21}+I_{22}+I_{23}.
\end{align*}

Since 
\begin{align*}
\ds I_{21} & \leq C \int_{t_1}^{t_2} \int_{\T^d}\int_0^1 \sig(t)| \alpha_{1,n}(t+\theta\eta, x+\delta)-f(x+\delta, m(t+\theta\eta,x+\delta))|\dd\theta\dd x\dd t\\
& \leq C \int_0^1\int_{t_1-\theta \eta}^{t_2-\theta \eta} \int_{\T^d} \sig^{-\theta\eta}(t)| \alpha_{1,n}(t, x)-f(x, m(t,x))|\dd x\dd t\dd\theta
\end{align*}
we obtain by Cauchy-Schwarz and \eqref{estifondBIS}: 
$$
 I_{21}\; \leq C \|\sig\|_{L^2}  (1/n+|\d|^2+\eta^2)^{1/2}.
$$
The term $I_{23}$ can be treated in the same way. For $I_{22}$, we have 
\begin{align*}
I_{22}&\leq C \|\sig\|_{L^2} \int_0^1 \left(\int_{t_1}^{t_2} \int_{\T^d} |f(x+\delta, m(t+\theta \eta,x+\delta))-f(x,m(t+\theta\eta,x))|^2\dd x\dd t\right)^{1/2}\dd\theta\\ 
&\leq C \|\sig\|_{L^2}|\delta|
\end{align*}
because $f(\cdot,m(\cdot,\cdot))$ is in $H^1_{\rm{loc}}((0,T)\times\T^d)$. 

Putting the above inequalities together gives \eqref{ineq.I}. 

\bigskip

{\it Step 3.} We now show that $p_n:= (\alpha_n-f(\cdot,\ov m))_+$ belongs to the space $L^2([t_1,t_2];BV(\T^d))$. Let us take a test function $\psi\in C_c^{\infty}((0,T)\times\T^d),$ $e\in\R^n$ with $|e|=1,$ $\eta>0$ small and let us set $\d:=\eta e.$ Let us estimate
\begin{align*}
\int_{t_1}^{t_2}\int_{\T^d}\sig^{-\eta}(t)&\frac{\psi^{-\eta}(t,x)-\psi^{-\eta}(t,x-\d)}{\eta}p_n(t,x)\dd x\dd t\\
&=\int_{t_1-\eta}^{t_2-\eta}\int_{\T^d}\sig(t)\frac{\psi(t,x)-\psi(t,x-\d)}{\eta}p_n^{\eta}(t,x)\dd x\dd t\\
&=\int_{t_1-\eta}^{t_2-\eta}\int_{\T^d}\sig(t)\psi(x)\frac{p_n^{\eta}(t,x)-p_n^{\eta}(t,x+\d)}{\eta}\dd x\dd t\\
&\le \|\psi\|_{L^{\infty}}\frac{1}{\eta}\int_{t_1-\eta}^{t_2-\eta}\int_{\T^d}\sig(t)|p_n^{\eta}(t,x)-p_n^{\eta}(t,x+\d)|\dd x\dd t\\
&\le \|\psi\|_{L^{\infty}}\frac{1}{\eta}X(\sig,1/n,\eta e,\eta).
\end{align*}
First let us recall that $p_n\weaks p$ as $n\to+\infty$ in $\fM([0,T]\times\T^d),$ which allows us to pass to the limit in the above inequality as $n\to+\infty$ and obtain
$$\int_{t_1}^{t_2}\int_{\T^d}\sig^{-\eta}(t)\frac{\psi^{-\eta}(t,x)-\psi^{-\eta}(t,x-\d)}{\eta}\dd p(t,x) \le \|\psi\|_{L^{\infty}}\frac{1}{\eta}X(\sig,0,\eta e,\eta).$$
Now sending $\eta\to 0$ and recalling \eqref{ineq.I} we have
$$\int_{t_1}^{t_2}\int_{\T^d}\sig(t)\nabla\psi(t,x)\cdot e\dd p(t,x)\le C\|\psi\|_{L^{\infty}}\|\sig\|_{L^2}.$$
Therefore we obtain that $p\in L^2([t_1,t_2];BV(\T^d))\hookrightarrow L^2([t_1,t_2];L^{d/(d-1)}(\T^d))$ and in particular, by the arbitrariness of $t_1$ and $t_2$ and by an injection we have $p\in L^{d/(d-1)}_{\rm{loc}}((0,T)\times\T^d).$

\bigskip

{\it Step 4.} Conclusion: as $0\leq \alpha_n\leq p_n+ f(\cdot,\ov m)$ and $(\alpha_n)$ converges to $\alpha$ defined by \eqref{def.alpha}, we have 
$0\leq \beta\leq p+f(\cdot,\ov m)$ in $(0,T)\times \T^d$. This proves that $\beta$ is absolutely continuous and belongs to  $L^{d/(d-1)}_{\rm{loc}}((0,T)\times\T^d).$
\end{proof}
\begin{remark}
Note that by the example provided in the Section \ref{sec:example} we have the sharpness of the above integrability result in the following sense: we cannot expect a bound for $p_n$ in  $L^{d/(d-1)}_{\rm{loc}}((0,T]\times\T^d)$, i.e., up to the final time, because of the occurrence of a possible jump at $t=T$. 
\end{remark}

\section{Nash equilibria for MFG with density constraints}\label{sec:nash}
Let us suppose in this section the additional assumptions {\bf{(HP1)}}, {\bf{(HP2)}} and {\bf{(HP3)}} as in Section~\ref{sec:regularity}. To define a proper notion of Nash equilibrium, we use the techniques for measures on paths, corresponding the trajectories of single agents. This will also allow us to clarify the meaning of the control problem \eqref{new_control}. The used machinery is inspired by \cite{af2} (Section 6) and also by \cite{carda} (Section 4.3) and \cite{carcarnaz} (Section 4). We remark also some similarities of this approach with the works modeling traffic congestion and Wardrop equilibria (see \cite{BraCarSan, CarJimSan}). 

\subsection{Density-constrained flows and a first optimality condition}
Let us recall that  $\Gamma$ denotes the set of absolutely continuous curves $\g:[0,T]\to\T^d$ and $\cP_2(\Gamma)$ the set of Borel probability measures $\tilde\et$  defined on $\Gamma$  such that 
$$\int_\Gamma\int_0^T|\dot\g(s)|^2\dd s\dd\tilde\et(\g)<+\infty.$$
We call $\tilde \et$ an {\it almost density-constrained flow} if there exists $C=C(\tilde\et)>0$ such that $0\le \tilde m_t\le C(\tilde\et)$ a.e. in $\T^d$ for all $t\in[0,T],$ where $\tilde m_t:=(e_t)_\#\tilde\et.$ If $C(\tilde\et)\le\ov m$ (the density constraint, given by our model) then we call $\tilde\et$ a \emph{density-constrained flow}.  Let us recall moreover that we use the definition of the Lagrangian as $L(x,v)=H^*(x,-v).$ 

In the whole section we consider a solution  $(u,m,\b,\b_T)$ of the MFG system \eqref{MFGdensity}. By Theorem \ref{thm:sol} and Theorem \ref{regularity_pressure}  this corresponds to $(u,\a)$ and $(m,w)$ solutions of Problem \ref{pr:relaxed} and Problem \ref{pr:dual} respectively, where
$$\alpha= f(\cdot, m) \dd x\dd t+\beta\dd x\dd t+\beta_T \dd(\d_T\otimes {\mathcal H}^{d}\mres \T^d)\;\;\;{\rm{and}}\;\;\; w= -mD_pH(x, Du).$$ 

Let us state the following results (in the spirit of Lemma 4.6.-4.8. from \cite{carda}) which characterize the density-constrained flows.
\begin{lemma}\label{lem:densflow}
Let $\tilde\et\in\cP_2(\Gamma)$ be an almost density-constrained flow and set $\tilde m_t:=(e_t)_\#\tilde\et.$ Then 
\begin{itemize}
\item[(i)] for all $0\le t_1<t_2\leq T$ we have
\begin{align*}
\int_{\T^d}u(t_1^+,x)\tilde m(t_1,x)\dd x &\le \int_{\T^d}u(t_2^-,x)\tilde m(t_2,x)\dd x+\int_{\Gamma}\int_{t_1}^{t_2}L(\g(t),\dot\g(t))\dd t\dd\tilde\et(\g)\\
&+\int_{t_1}^{t_2}\int_{\T^d}\alpha(t,x)\tilde m(t,x)\dd x\dd t.
\end{align*}
(we recall $\alpha(t,x):=f(x, m(t,x))+\b(t,x)$)
\item[(ii)] In particular, for all $0\le t_1<T$
\begin{align*}
\int_{\T^d}u(t_1^+,x)\tilde m(t_1,x)\dd x &\le \int_{\T^d}(g(x)+\b_T(x))\tilde m(T,x)\dd x+\int_{\Gamma}\int_0^TL(\g(t),\dot\g(t))\dd t\dd\tilde\et(\g)\\
&+\int_0^T\int_{\T^d}\alpha(t,x)\tilde m(t,x)\dd x\dd t.
\end{align*}
\end{itemize} 
\end{lemma}

\begin{proof} 
Let us recall that $u$ satisfies, in the sense of measures,  
$$
-\partial_t u +H(x,Du)\leq \alpha \qquad \rm{in}\; (0,T)\times \T^d,
$$
where $\alpha$ belongs to $L^{d/(d-1)}_{loc}((0,T)\times\T^d)$ thanks to Theorem \ref{regularity_pressure}. 
If we regularize $u$ into $u_n$ and $\a$ into $\a_n$ by convolution (with a compact support in $B_{1/n}(0)$), we obtain 
$$
-\partial_t u_n +H(x,Du_n)\leq \alpha_n+ r_n \qquad \rm{in}\; (1/n,T-1/n)\times \T^d,
$$
where 
$$
r_n(t,x)= H(x,Du_n(t,x))-H(\cdot, Du)\star \rho_n(t,x).
$$
Note that $(r_n)$ tends to $0$ in $L^1((0,T)\times \T^d)$. By the way, if $H$ is independent of $x$, one also has $r_n\le 0$. Let us fix $0<t_1<t_2<T$ and $n$ large. 
Now for any $\g\in H^1([0,T])$ we have
\begin{align}\label{derive}
\frac{\dd}{\dd t}\left(u_n(t,\g(t))-\int_t^TL(\g(s),\dot\g(s))\dd s\right)&\ge\partial_t u_n(t,\g(t))-H(\g(t),Du_n(t,\g(t)))\nonumber\\
&\ge-\a_n(t,\g(t))-r_n(t,\g(t)).
\end{align}
Integrating this inequality on $[t_1,t_2],$ then over $\Gamma$ w.r.t. $\tilde\et,$ we obtain
\begin{align*}
\int_{\T^d}u_n(t_1,x)\tilde m(t_1,x)\dd x&\le \int_{\T^d}u_n(t_2,x)\tilde m(t_2,x)\dd x +\int_\Gamma\int_{t_1}^{t_2}L(\g(t),\dot\g(t))\dd t\dd\tilde\et(\g)\\
&+\int_{\T^d}\int_{t_1}^{t_2}[\a_n(t,x)+r_n(t,x)]\tilde m(t,x)\dd t\dd x.
\end{align*} 
We recall the fact that $\tilde m\in L^\infty([0,T]\times\T^d).$ Since $(u_n)$ strongly  converges in $L^1$ to $u\in BV([0,T]\times\T^d)$, we have the existence of $J\subset(0,T)$ of full measure such that  for every $t_1,t_2\in J, t_1<t_2,$ the first two integrals pass to the limit as $n\to+\infty.$ By the strong convergence in $L^{d/(d-1)}([t_1,t_2]\times\T^d)$ of $(\a_n)$ to $\a$ and in $L^1([t_1,t_2]\times\T^d)$ of $(r_n)$ to 0, we can pass to the limit as $n\to+\infty$ is the last integral as well. So, for a.e. $0<t_1<t_2<T$, we have
\begin{align*}
\int_{\T^d}u(t_1,x)\tilde m(t_1,x)\dd x&\le \int_{\T^d}u(t_2,x)\tilde m(t_2,x)\dd x +\int_\Gamma\int_{t_1}^{t_2}L(\g(t),\dot\g(t))\dd t\dd\tilde\et(\g)\\
&+\int_{\T^d}\int_{t_1}^{t_2}\a(t,x)\tilde m(t,x)\dd t\dd x.
\end{align*} 
In order to show that the inequality holds for any $t_1<t_2$, let us now check that 
$$
\lim_{t'\to t^\pm} \int_{\T^d}u((t')^+,x)\tilde m(t,x)\dd x = \int_{\T^d}u(t^\pm,x)\tilde m(t,x)\dd x, 
$$
where $u((t')^\pm,\cdot)$ is understood in the sense of trace and $\tilde m(t,\cdot)$ is the (bounded) density of the continuous representative of the map $t\mapsto  \tilde m(t,\cdot)\dd x$ (for the $L^\infty$ weak$-\star$ convergence). The above limit basically follows from the trace properties of $BV$ functions, but for the sake of completeness let us sketch it below. 
Let $u_n$ be a standard mollification in space of $u$. As $u$ is in BV, $u((t')^+,\cdot)$ converges in $L^1$ to $u(t^\pm,\cdot)$ as $t'\to t^\pm$, so that $u_n((t')^+,\cdot)$ uniformly converges to $u_n(t^\pm,\cdot)$. Let us write $\ds \int_{\T^d}u((t')^+,x)\tilde m(t',x)\dd x$ as 
\be\label{kujzvbozsdcj}
\int_{\T^d}u_n((t')^+,x)\tilde m(t',x)\dd x
+ 
\int_{\T^d}(u((t')^+,x)-u_n((t')^+,x))\tilde m(t',x))\dd x.
\ee
By uniform convergence of $u_n((t')^+,\cdot)$, the first term in \eqref{kujzvbozsdcj} converges to $\ds \int_{\T^d}u_n(t^\pm,x)\tilde m(t,x)\dd x$, which is arbitrary close to $\ds \int_{\T^d}u(t^\pm,x)\tilde m(t,x)\dd x$ for $n$ large. As for the second term in \eqref{kujzvbozsdcj}, it is bounded by $ \|u((t')^+,\cdot)-u_n((t')^+,\cdot)\|_{L^1}   \| m\|_{L^\infty}$, which, by $L^1$ convergence of $u((t')^+,\cdot)$ to $u(t^\pm,\cdot)$, tends to $0$  uniformly in $t'$. This proves  (i). 

For (ii), we just apply (i) for $t_2=T$, since $u(T^-,\cdot)= g+ \beta_T$. 
\end{proof}

\begin{definition}\label{def:optdenflow}
We say that an $\et\in\cP_2(\Gamma)$ is an {\emph{optimal density-constrained flow}} associated with the solution $(u,m,\b,\b_T)$ if $ m(t,\cdot)=(e_t)_\#\et$, 
for all $t\in[0,T]$ and the following energy equality holds
\begin{align*}
\int_{\T^d}u(0^+,x)m_0(x)\dd x &= \int_{\T^d}g(x)m(T,x)\dd x+ \ov m\int_{\T^d}\b_T\dd x+\int_{\Gamma}\int_0^TL(\g(t),\dot\g(t))\dd t\dd\et(\g)\\
&+\int_0^T\int_{\T^d}\left(f(x,m(t,x))+\b(t,x)\right)m(t,x)\dd x\dd t.
\end{align*}
\end{definition}

Note that the above definition is the reformulation in terms of density-constrained flows of the energy equality from Definition \ref{def:solution} point \eqref{point4}.  
\begin{remark}
Let us observe that for an optimal density-constrained flow $\et,$ the energy equality in Definition \ref{def:optdenflow} holds for any $0\le t_1<t_2\le T$ as well, i.e.
\begin{align}
\int_{\T^d}u(t_1^+,x)m(t_1,x)\dd x &= \int_{\T^d}u(t_2^-,x)m(t_2,x)\dd x+\int_{\Gamma}\int_{t_1}^{t_2}L(\g(t),\dot\g(t))\dd t\dd\et(\g)\nonumber\\
&+\int_{t_1}^{t_2}\int_{\T^d}\left(f(x,m(t,x))+\b(t,x)\right)m(t,x)\dd x\dd t.\label{opt_local}
\end{align}
This can be easily deduced using the inequalities from Lemma \ref{lem:densflow} three times on the intervals $[0,t_1],[t_1,t_2]$ and $[t_2,T]$ together with the global equality from Definition \ref{def:optdenflow} and the fact that we have $(\partial_t u)^{s}$ is a non-negative measure (by the fact that $-\partial_t u+H(\cdot,Du)\le \a$ and $\a$ does not have singular part in $(0,T)\times\T^d$), i.e. one has always $u(t^-,\cdot)\le u(t^+,\cdot)$ a.e. for all $t\in(0,T).$

Identity \eqref{opt_local} implies also that $(\partial_tu)^s=0$ on the support of $m,$ more precisely 
\begin{equation}\label{u+=u}
\int_{\T^d}u(t^+,x)m(t,x)\dd x=\int_{\T^d}u(t^-,x)m(t,x)\dd x,
\end{equation}
for all $t\in(0,T).$
\end{remark}

The following proposition gives the existence result for an optimal density-constrained flow $\et.$ 
\begin{proposition}\label{prop:ex_flow}
There exists at least one optimal density-constrained flow $\et\in\cP_2(\Gamma)$ in the sense of the Definition \ref{def:optdenflow}.
\end{proposition}
\begin{proof}
The proof uses the same construction and goes along the same lines as in \cite{carda}. Nevertheless, we discuss the main steps here.

We construct a family $(\et_\e)_{\e>0}$ of density-constrained flows by
$$\int_\Gamma\Psi(\g)\dd\et_\e(\g):=\int_{\T^d}\Psi(X_\e^x)m_0(x)\dd x,$$
for any bounded and continuous map $\Psi:\Gamma\to\R,$ where $X_\e^x$ is the solution of the Cauchy problem
$$
\left\{
\begin{array}{ll}
\ds \dot x(t)=\frac{w_\e(t,x(t))}{m_\e(t,x(t))}, & {\rm{a.e\ in\ }}[0,T],\\[10pt]
x(0)=x,
\end{array}
\right.
$$
$(m_\e,w_\e)$ being a standard mollification of $(m,w)$ such that $0<m_\e\le\ov m.$ One easily checks  that $m_\e(t,\cdot)=(e_t)_\#\et_\e.$

Using Lemma 4.7. from \cite{carda} we obtain that the family  $(\et_\e)_{\e>0}$ is tight. Denoting by $\et$ the limit of a suitable subsequence of such a family, this is an optimal density-constrained flow in the sense of Definition \ref{def:optdenflow}. The proof of this statement goes exactly as for Lemma 4.8. in \cite{carda}, using the equality \eqref{point4} from Definition \ref{def:solution} and the inequality $(ii)$ from Lemma \ref{lem:densflow}.
\end{proof}
\subsection{Optimality conditions on the level of single agent trajectories}
In this subsection our aim is to show that the optimal density-constrained flows are actually concentrated on paths which are optimal (in some weak sense) for the control  problem \eqref{new_control} (see Definition \ref{def:minpath}). We will show that they satisfy a weak dynamic programming principle. 

Let us recall that $\b\in L^2_{\rm{loc}}((0,T);BV(\T^d))$ and $\b_T\in L^1(\T^d).$ In order to handle the evaluation of $\b$ along single agent paths we shall work with specific representative of it (which is defined everywhere in $\T^d$).

For an $L^1_{\rm{loc}}$ function $h:\T^d\to\R$ we define the specific representative of $h$ by 
\begin{equation}\label{representative}
\hat h(x):=\limsup_{\e\downarrow 0}h_\e(x),
\end{equation}
where  
$$h_\e(x):=\int_{\R^d}h(x+\e y)\rho(y)\dd y$$
and $\rho$ being the heat kernel  
\begin{equation}\label{heat_k}
\rho(y):=(2\pi)^{-d/2}e^{-|y|^2/2}. 
\end{equation}
We use this specific regularization via the heat kernel because of the semigroup property $(h_\e)_{\e'}=h_{\e+\e'}$ we shall profit on later. 

To treat passages to limit (in the regularization, as $\e\downarrow 0$, similarly as in Section 6 from  \cite{af2}) we will need some uniform point-wise bounds on $\b_\e$, hence we shall use the properties of the Hardy-Littlewood-type maximal function defined with the help of the heat kernel \eqref{heat_k}. Thus for any $h\in L^1(\T^d)$ we set
$$(Mh)(x):=\sup_{\e>0}\int_{\R^d}|h(x+\e y)|\rho(y)\dd y.$$ 
Let us state some basic properties of the maximal functional $M$ that we will use in our setting. First because of the semigroup property we have 
$$Mh_\e=\sup_{\e'>0}|h_\e|_{\e'}\le\sup_{\tilde\e>0}|h|_{\tilde\e}=Mh.$$
Secondly it is well-known that $M$ leaves invariant any $L^p$ space with $1<p\le+\infty$ and there exists $C_p>0$ such that 
$$\|Mh\|_{L^p(\T^d)}\le C_p\|h\|_{L^p(\T^d)}.$$
Let us recall that by Theorem \ref{regularity_pressure} we have that $M\b\in L_{\rm{loc}}^{d/(d-1)}((0,T)\times\T^d)\hookrightarrow L_{\rm{loc}}^{1}((0,T)\times\T^d).$ The integrability property we need is only $M\beta\in L^1_{\rm{loc}}((0,T)\times\T^d)$, but to guarantee this, $\beta\in L^1_{\rm{loc}}((0,T)\times\T^d)$ is not enough.  

 As usual we set $\a(t,x):=f(x,m(t,x))+\b(t,x)$ and we use its representative $\hat \a$ (obtained as in \eqref{representative}).  

\begin{definition}\label{def:minpath}
Given $0<t_1<t_2<T$, we say that a path $\g\in H^1([0,T];\T^d)$ with $M\hat \a(\cdot,\g)\in L_{\rm{loc}}^1((0,T))$ is \emph{minimizing} on the time interval $[t_1,t_2]$ in the problem \eqref{new_control} if we have
\begin{align}\label{defi minimizing single agent}
\hat u(t_2,\g(t_2))+\int_{t_1}^{t_2}L(\g(t),\dot\g(t))+\hat \a&(t,\g(t))\dd t\le\hat u(t_2,\g(t_2)+\omega(t_2))+\\
&+\int_{t_1}^{t_2}L(\g(t)+\omega(t),\dot\g(t)+\dot\omega(t))+\hat \a(t,\g(t)+\omega(t))\dd t,
\end{align}
for all $\omega\in H^1([t_1,t_2];\T^d)$ such that $\omega(t_1)=0$  and  $M\hat \a(\cdot,\g+\omega)\in L^1([t_1,t_2]).$
\end{definition}

\begin{remark}
Let us notice that for any density-constrained flow $\tilde\et$ the integrability property $M\hat \a(\cdot,\g)\in L_{\rm{loc}}^1((0,T))$ is natural, since it is satisfied $\tilde\et$-a.e., if $M\hat \a\in L_{\rm{loc}}^1((0,T)\times\T^d).$ Indeed, we have
$$\int_\Gamma\int_{t_1}^{t_2}M\hat \a(t,\g(t))\dd t\dd\tilde\et(\g)=\int_{t_1}^{t_2}\int_{\T^d}M\hat \a(t,x)\tilde m(t,x)\dd x\dd t<+\infty,$$
for all $0<t_1<t_2<T,$ where $\tilde m(t,\cdot)\dd x=(e_t)_\#\tilde\et.$
\end{remark}

\begin{theorem}\label{thm:minpath}
For any $0<t_1<t_2<T$, any optimal density-constrained flow $\et$ is concentrated on minimizing paths on the time interval $[t_1,t_2]$ for the problem \eqref{new_control} in the sense of the Definition \ref{def:minpath}.
\end{theorem}

\begin{proof} We follow here Ambrosio-Figalli \cite{af2}. Let us take an optimal density-constrained flow $\et$ given by Proposition \ref{prop:ex_flow}, fix $0<t_1<t_2<T$ and $y\in\T^d$, take $\omega\in H^1([t_1,t_2];\T^d)$ with $\omega(t_1)=0$ and $\chi\in C_c^1((0,T);[0,1])$ with $\chi>0$ on $(t_1,t_2]$ and $\chi(t_1)=0$ a smooth cut-off function. Let us take a Borel subset $E\subset\Gamma$ such that $\et(E)$ is positive.  For $\e>0$ and $y\in\T^d$ we introduce the map $T_{\e,y}:\Gamma\to\Gamma$ by
$$T_{\e,y}(\g):=\left\{
\begin{array}{ll}
\g, & {\rm{if}}\;\g\notin E,\\
\g+\omega+\e\chi y, & {\rm{if}}\;\g\in E.
\end{array}
\right.
$$ 
Now let us define $\et_{\e,y}:=(T_{\e,y})_\#\et,$ which in particular is an admissible density-constrained flow satisfying the inequalities from Lemma \ref{lem:densflow}. In addition let us remark that $(e_{t_1})_\#\et_{\e,y}=(e_{t_1})_\#\et= m(t_1,\cdot )\dd x$.

Using the inequality $(i)$ from Lemma \ref{lem:densflow} for $\et_{\e,y}$ (on the time interval $[t_1,t_2]$) and the equality \eqref{opt_local} for $\et$ (on the same interval $[t_1,t_2]$) we obtain
\begin{align*}
\int_{E} \Big[\hat u&(t_2^+,\g(t_2))+ \int_{t_1}^{t_2}L(\g(t),\dot\g(t))+\hat \a(t,\g(t)) \dd t\Big]\dd\et(\g)\le\\
&\int_{E}\Big[\hat u(t_2^-,\g(t_2)+\omega(t_2) +\e\chi(t_2)y)+\int_{t_1}^{t_2}L(\g(t)+\omega(t)+\e\chi(t)y,\dot\g(t)+\dot\omega(t)+\e\dot\chi(t)y)\dd t\Big]\dd\et(\g)\\
&+\int_{E}\int_{t_1}^{t_2}\hat \a(t,\g(t)+\omega(t)+\e\chi(t)y)\dd t\dd\et(\g).
\end{align*}
where we are allowed to use any representative of $u$ and $\a,$ thus we use the specially constructed ones $\hat u$ and $\hat \a.$ 
Let us average this last inequality w.r.t. the variable $y$ using the kernel $\rho$ introduced in \eqref{heat_k}. We obtain
\begin{multline*}
\int_{E}  \Big[\hat u (t_2^+,\g(t_2))+\int_{t_1}^{t_2}L(\g(t),\dot\g(t))+\hat \a(t,\g(t)) \dd t\Big]\dd\et(\g)\le\\
\int_{E}\int_{\R^d} \Big[\hat u(t_2^-,\g(t_2)+\omega(t_2)+\e\chi(t_2)y)+\int_{t_1}^{t_2}L(\g(t)+\omega(t)+\e\chi(t)y,\dot\g(t)+\dot\omega(t)+\e\dot\chi(t)y)\dd t\Big]\rho(y) \dd y \dd\et(\g)\\
+\int_{E}\int_{t_1}^{t_2}\hat \a_{\e\chi(t)}(t,\g(t)+\omega(t))\dd t\dd\et(\g).
\end{multline*}
Now choosing $\mathcal{D}\subset H^1([t_1,t_2];\T^d)$ a dense subset with $\omega(t_1)=0$ for all $\omega\in\mathcal{D}$, by the arbitrariness of $E$ for $\et$-almost every curve $\g\in\Gamma$ we deduce that
\begin{multline*}
 \hat u(t_2^+,\g(t_2))+\int_{t_1}^{t_2}L(\g(t),\dot\g(t))+\hat \a(t,\g(t)) \dd t\\
 \le \int_{\R^d} \int_{t_1}^{t_2}L(\g(t)+\omega(t)+\e\chi(t)y,\dot\g(t)+\dot\omega(t)+\e\dot\chi(t)y)\rho(y)\dd t\dd y \\
+ \hat u_{\e\chi(t_2)}(t_2^-,\g(t_2)+\omega(t_2))  +\int_{t_1}^{t_2}\hat \a_{\e\chi(t)}(t,\g(t)+\omega(t))\dd t,
\end{multline*}
for all $\omega\in\mathcal{D}$ and $\e=1/n.$ By a density argument the above inequality holds for any $\omega\in H^1([t_1,t_2];\T^d)$ with $\omega(t_1)=0.$ We finally let  $\e \downarrow 0$. As $M\hat\a(t,\g+\omega)\in L^1([t_1,t_2])$ and using the domination $|\a_\e|\le M\hat\a$, we can pass to the limit in the last term of the above inequality. By the dominate convergence theorem we can also pass to the limit in the first term thanks to the growth property and the continuity of $L$. In this way we obtain the inequality \eqref{defi minimizing single agent} with $\hat u(t_2^+,\g(t_2))$ instead of $\hat u(t_2,\g(t_2))$ and $\hat u(t_2^-,\g(t_2)+\omega(t_2))$. To conclude, it is sufficient to use $u(t_2^-,\cdot)\le u(t_2,\cdot) $ (a consequence of $\partial_t^s u\ge 0$) and $u(t_2^+,\cdot)= u(t_2,\cdot) $ $m_{t_2}$-a.e. (given by \eqref{u+=u}).
\end{proof}

\begin{remark}
The global version of Theorem \ref{thm:minpath} (to arrive up to the initial time $0$ and the final time $T$) remains an open question. This is mainly due to the local integrability property for the additional price $\b\in L_{\rm{loc}}^2((0,T);BV(\T^d))$ we are aware of for the moment. Let us remark that an integrability property $\b\in L^{1}([0,T];L^{1+\e}(\T^d))$ for some $\e>0$ would be enough to conclude in the global version.
\end{remark}

\vspace{0.5cm}
The {\it notion of Nash equilibria} has now a clearer formulation. Since we are able to give a weak meaning for the optimization problem along single agent trajectories, a solution $(u,m,\b,\b_T)$ of the MFG system with density constraints gives the following notion of equilibrium. 

\begin{definition}[Local weak Nash equilibria]\label{def:nash}
Let $(u,m,\b,\b_T)$ be a solution of the MFG system with density constraints in the sense of Definition \ref{def:solution} on $[0,T]\times\T^d$. We say that $(m,\b,\b_T)$ is a ${\rm{local\ weak\ Nash\ equilibrium}}$,  if there exists an optimal density-constrained flow $\et\in\cP_2(\Gamma)$ in the sense of Definition \ref{def:optdenflow} (constructed with the help of $(m,\b,\b_T)$) which is concentrated on locally minimizing paths for Problem \eqref{new_control} in the sense of  Definition \ref{def:minpath}. In particular one has that $m_t=(e_t)_\#\et$ and $0\le m_t\le \ov m$ a.e. in $\T^d$ for all $t\in[0,T]$.   
\end{definition}
\begin{remark}
Let us remark that by Proposition \ref{prop:ex_flow} and Theorem \ref{thm:minpath} for any solution $(u,m,\b,\b_T)$ for the MFG system with density constraints obtained with the additional assumptions {\bf{(HP1)}}, {\bf{(HP2)}} and {\bf{(HP3)}} the triple $(m,\b,\b_T)$ is always a local weak Nash equilibrium in the sense of the above definition.
\end{remark}

\subsection{The case without density constraint}
Let us have a few words on the Nash equilibrium and on the optimality condition on the level of single agent trajectories in the case when we do not impose density constraints. More precisely, our aim is to clarify Remark 4.9. from \cite{carda}.    

Let us recall that in Section 4.3. from \cite{carda} it was considered a class of flows $\tilde\et\in\cP_{r'}(\T^d)$ such that $\tilde m\in L^q([0,T]\times\T^d)$ where $\tilde m_t:=(e_t)_\#\tilde\et,$ where $r'>1$ is the growth of the Lagrangian $L$ in the velocity variable, while $q-1$ (where $q>1$) is the growth of the continuous coupling $f$ in the second variable. 
Because of this growth condition and since $m\in L^q([0,T]\times\T^d)$ we have first that $\a(t,x):=f(x,m(t,x))\in L^{q'}([0,T]\times\T^d).$
Moreover Lemma \ref{lem:densflow} and Proposition \ref{prop:ex_flow} hold with $\b\equiv0$ and $\b_T\equiv 0,$ since we did not impose any density constraint (see the corresponding Lemma 4.6-4.8 from \cite{carda}). 

The difference, compared to our analysis in the previous section, is that we can consider globally minimizing paths in Definition \ref{def:minpath}. More precisely, by the global integrability property of $\hat \a,$ and hence $M\hat \a\in L^{q'}([0,T]\times\T^d)$ we allow curves $\g\in W^{1,r'}([0,T])$ (and their variations) such that $M\hat \a(\cdot,\g)\in L^{q'}([0,T]).$ This is once again a natural class, since for any flow $\tilde\et$, with the above described properties, satisfies that 
$$\int_\Gamma\int_{t_1}^{t_2}M\hat \a(t,\g(t))\dd t\dd\tilde\et(\g)=\int_{t_1}^{t_2}\int_{\T^d}M\hat \a(t,x)\tilde m(t,x)\dd x\dd t<+\infty,$$
for all $0\le t_1<t_2\le T,$ since $M\hat \a\in L^{q'}([0,T]\times\T^d)$ and $\tilde m\in L^{q}([0,T]\times\T^d)$ where $\tilde m_t=(e_t)_\#\tilde\et.$ 

By these observations  in the statement of Theorem \ref{thm:minpath} one can change now the word ``locally'' to ``globally'' and the proof goes along the same lines.
\bigskip

%

%
%

%
%
%
%

\end{document}